\documentclass[10pt]{amsart}
\usepackage{}
\usepackage{amsfonts}
\usepackage{amsfonts,latexsym,rawfonts,amsmath,amssymb,amsthm}
\usepackage[plainpages=false]{hyperref}
\usepackage{mathrsfs}
\RequirePackage{amsmath} \RequirePackage{amssymb}
\usepackage{enumitem}
\usepackage{fancyhdr}
\usepackage{amscd,latexsym,amsthm,amsfonts,amssymb,amsmath,amsxtra}
\usepackage{amscd,mathrsfs,latexsym,amsmath,amsthm,amssymb,amsxtra,}
\usepackage{latexsym,amsmath,amsthm,amssymb,amsxtra,mathrsfs}
\usepackage[all]{xy}
\usepackage{longtable}
\usepackage{verbatim}
\usepackage{bm}
\usepackage{mathtools}
\usepackage{color}

\makeatletter
\renewcommand{\section}{\@startsection{section}{1}%
  \z@{.7\linespacing\@plus\linespacing}{.5\linespacing}%
  {\normalfont\bfseries\centering}}
\renewcommand{\@secnumfont}{\bfseries}
\makeatother

\numberwithin{equation}{section}

\newcommand{\dd}{\,dV_g}

\newcommand{\beq}{\begin{equation}}
\newcommand{\eeq}{\end{equation}}
\newcommand{\beqs}{\begin{eqnarray*}}
\newcommand{\eeqs}{\end{eqnarray*}}
\newcommand{\beqn}{\begin{eqnarray}}
\newcommand{\eeqn}{\end{eqnarray}}
\newcommand{\beqa}{\begin{array}}
\newcommand{\eeqa}{\end{array}}

\def\lra{\longrightarrow}

\def\bc{\begin{center}}
\def\ec{\end{center}}

\def\begeq{\begin{equation}}
\def\endeq{\end{equation}}
\def\and{\quad{\rm and}\quad}

\let\lra=\longrightarrow

\def\mapright\#1{\,\smash{\mathop{\lra}\limits^{\#1}}\,}

\def\Ric{\mathrm{Ric}}

\newtheorem{prop}{Proposition}[section]

\newtheorem{rem}[prop]{Remark}

\newtheorem{conj}[prop]{Conjecture}

\allowdisplaybreaks[1]
\numberwithin{equation}{section}
\newtheorem{theorem}{Theorem}[section]
\newtheorem{lemma}[theorem]{Lemma}
\newtheorem{proposition}[theorem]{Proposition}
\newtheorem{corollary}[theorem]{Corollary}
\DeclareMathOperator{\Div}{div}
\DeclareMathOperator{\Dom}{Dom}
\newcommand{\norm}[1]{\left\lVert#1\right\rVert}
\newcommand{\ip}[2]{\left\langle#1,#2\right\rangle}
\title{Complete K\"ahler Manifolds with pinched Ricci curvature}

\author{Yuxing Deng$^{*}$}
\address{School of Mathematics and Statistics, Beijing Institute of Technology, Beijing, 100081, China}
\email{6120180026@bit.edu.cn}

\author{Ruyi Ji}
\address{School of Mathematics and Statistics, Beijing Institute of Technology, Beijing, 100081, China}
\email{3120266150@bit.edu.cn}

\thanks {* Supported by National Key R$\&$D Program of China 2022YFA1007600 and NSFC No.12571059.}

\date{\today}

\begin{document}
\begin{abstract}
In this paper, we prove that a connected complete K\"ahler manifold of complex
dimension at least two satisfying $\Ric\geq\varepsilon Rg\geq0$
for a fixed $\varepsilon>0$ is compact or Ricci-flat.
\end{abstract}
\maketitle

\section{Introduction}\label{sec:introduction}

The Bonnet--Myers theorem gives compactness when the Ricci tensor
has a uniform positive lower bound. A natural question is whether
one may replace that lower bound by a pointwise comparison of the
Ricci eigenvalues. More precisely, it is natural to make the following conjecture.
\begin{conj}[Hamilton-Ni]\label{conj-1}
    If $(M,g)$ is  an $n(\ge 3)$-dimensional  complete
Riemannian manifold satisfies
\begin{equation}\label{eq:pinching}
 \Ric\geq\varepsilon Rg\geq0,
 \quad\text{for a fixed constant }\varepsilon>0,
\end{equation}
then it must be compact or Ricci-flat.
\end{conj}

Hamilton proposed Conjecture \ref{conj-1} for $n=3$ and Ni proposed Conjecture \ref{conj-1} for $n\ge3$\cite{Ni-05,Ni}. In dimension three, Chen-Zhu proved Conjecture \ref{conj-1} by assuming bounded and nonnegative sectional curvature \cite{ChenZhu}. They also proved the compactness for manifolds with  a stronger pinching condition in higher dimensions. In dimension three, Lott  proved Conjecture \ref{conj-1} on manifolds with  quadratic curvature decay ~\cite{Lott}. He also provided another proof for Chen-Zhu's result in dimension three. Later on, Deruelle-Schulze-Simon proved Conjecture \ref{conj-1} for $n=3$ by assuming the curvature is bounded~\cite{DSS}.  Lee and Topping proved Conjecture \ref{conj-1} for $n=3$~\cite{LeeTopping}. Recently, Deruelle-Lee-Schulze-Simon-Topping proved the compact-or-flat alternative under the PIC1 pinching condition in all dimensions~\cite{DLSST}. 
 All these results use the Ricci flow. Very recently, Xia and Zheng studied Conjecture \ref{conj-1} for manifolds with boundary in dimension three\cite{XZ}.


In the K\"ahler case, earlier K\"ahler vanishing results of Ni--Shi--Tam~\cite{NST}
treated pinched nonnegative closed $(1,1)$-forms under additional
curvature and decay assumptions. Zhu and the first named author proved the compact-or-flat alternative on expanding K\"ahler Ricci soliton with nonnegatively pinched Ricci curvature \cite{DZ}.   More recently,
Liu--Ni~\cite{LiuNi} proved flatness under nonnegative bisectional
curvature and Euclidean volume growth, and Ricci-flatness under
nonnegative Ricci curvature, Euclidean volume growth, and quadratic
curvature decay. Their work also allows a logarithmically degenerating
pinching constant in these volume-growth results. Yin~\cite{Yin}
removed the volume-growth assumption for uniform pinching under
nonnegative holomorphic bisectional curvature. The theorem considered
here retains global uniform Ricci pinching and 
additional assumptions on bisectional curvature and volume growth.

In this paper, we prove Conjecture \ref{conj-1}  when the manifold is K\"ahler.

\begin{theorem}\label{thm:main}
Let $(M^{2m},g,J)$ be a connected complete K\"ahler manifold of complex dimension $m\geq2$. If
\eqref{eq:pinching} holds for a fixed $\varepsilon>0$, then $M$ is
compact or Ricci-flat.
\end{theorem}

Assume that $M$ is noncompact. A localized
quadratic estimate, based on the contracted Bianchi identity, a
K\"ahler Hessian identity and spectral theorem, first gives
\[
 \int_M\eta^2R^2\dd
 \leq C\int_M R|\nabla\eta|^2\dd.
\]
This is an estimate with a compactly supported source, so it does not
presuppose $R\in L^2(M)$. In complex dimension two, iteration inside
a fixed outer ball gives this integrability, and an annular cutoff
then forces $R\equiv0$; see Corollary~\ref{cor:kahler-surface}.
For higher dimensions, we obtain estimates for higher powers by using
the closed forms
\[
 (\rho+\delta\omega)^k\wedge\omega^{m-k-2},
 \qquad 1\leq k\leq m-2,\qquad \delta>0,
\]
where $\omega$ and $\rho$ are the K\"ahler and Ricci forms.
The mixed Hodge--Riemann relation gives a strict absorption factor
$1-2m\varepsilon$ on the corresponding primitive subspace.
Each mixed form also defines a symmetric elliptic operator in a
weighted $L^2$ space. The main analytic point is to keep the resolvent parameter positive
while using a previously established lower-order estimate. On the
same fixed manifold, an estimate of order $k$ then yields one of
order $k+2$. The quadratic estimate supplies the cubic case
separately, and the two induction chains give, for $2\leq p\leq m$,
\begin{equation}\label{eq:intro-hierarchy}
 \int_M\eta^2R^p\dd
 \leq C_{m,\varepsilon,p}\left(
 \delta^{-1}\int_M(R+\delta)^p|\nabla\eta|^2\dd
 +\delta^p\int_M\eta^2\dd\right).
\end{equation}
The constants are independent of $\delta$ and $\eta$.
At the critical exponent $p=m$, an iteration inside a fixed outer
ball and Bishop--Gromov comparison imply $R\in L^m(M)$.
A final annular application of the same estimate forces this integral
to vanish. Thus the proof uses $L^2$ operator theory for auxiliary
functions, while deriving all needed integrability of the curvature.

The paper is organized as following. Section~\ref{sec:preliminaries} records the conventions, operator
preliminaries, and spectral measures.
Section~\ref{sec:resolvent-estimates} establishes the quadratic
estimate and proves the surface case.
Section~\ref{sec:mixed-ricci} develops the mixed forms and weighted
energy identities, and proves the localized estimates for powers
of the scalar curvature. Section~\ref{sec:proof} uses these estimates
to establish critical global integrability and complete the proof
of the main theorem. Appendices~\ref{app:approximation} and~\ref{app:mixed-hr}
collect the approximation argument and the pointwise mixed
Hodge--Riemann relation.

\medskip
\noindent\textbf{Declaration of AI use.}
The authors used ChatGPT for exploratory calculations in Lemma 4.1 and language editing during the preparation of the manuscript.
The authors take full responsibility for
the results of this paper.

\section{Preliminaries}\label{sec:preliminaries}

We work on the noncompact manifold in Theorem~\ref{thm:main}.
Tracing \eqref{eq:pinching} settles the case
$\varepsilon>1/(2m)$ immediately. Hence we assume
\begin{equation}\label{eq:q}
 0<\varepsilon\leq\frac1{2m},\qquad
 q=1-2m\varepsilon\in[0,1),\qquad c_m=\frac{m-1}{2m}.
\end{equation}
Throughout the curvature argument, $m$ and $\varepsilon>0$ are
fixed. Constants denoted by $C$ or $c>0$ may change from line to
line and depend only on $m,\varepsilon$, and the indicated integer
orders, unless an additional dependence is displayed. The constants
in the curvature estimates are independent of regularization
parameters, test functions, their supports, and cutoff radii.
No uniformity as $\varepsilon\downarrow0$ or $m\to\infty$ is
asserted. Local comparison factors, denoted by $C_{\mathrm{loc}}$,
and local regularity constants may depend on a fixed compact set
and the fixed coefficients. They are used with the regularization
parameters fixed and are distinct from the uniform constants in
the curvature estimates.

\subsection*{Conventions and operator preliminaries}
Here, we introduce some definitions, notation, and basic facts that will be used later.

We use $\Delta=\Div\nabla$ and $L=-\Delta$. All tensor norms and
contractions are induced by the real metric. Set
\begin{equation}\label{eq:conventions}
 \begin{gathered}
 \omega(X,Y)=g(JX,Y),\qquad \rho(X,Y)=\Ric(JX,Y),\\
 d^cu=-du\circ J,\qquad dV_g=\frac{\omega^m}{m!}.
 \end{gathered}
\end{equation}
Both $\omega$ and $\rho$ are closed.
The Lefschetz operator is $L_\omega\beta=\omega\wedge\beta$.
We denote its adjoint with respect to the pointwise metric on forms
by $\Lambda_\omega$; it lowers degree by two. Thus, for forms of
the appropriate degrees,
\[
 \ip{\Lambda_\omega\beta}{\gamma}
     =\ip{\beta}{\omega\wedge\gamma}.
\]
For a real two-form $\beta$, the contraction is
\[
 \Lambda_\omega\beta=\ip{\beta}{\omega}
        =\sum_{i=1}^m\beta(e_i,Je_i),\qquad
 \Lambda_\omega\omega=m.
\]
Here $(e_i,Je_i)_{i=1}^m$ is any real orthonormal frame adapted to $J$.
These expressions are independent of the chosen frame.
The $J$-invariant and $J$-anti-invariant parts of the Hessian are
\begin{equation}\label{eq:hessian}
 H^\pm u(X,Y)=\frac12\bigl(\nabla^2u(X,Y)
                  \pm\nabla^2u(JX,JY)\bigr),\qquad
 H^+_0u=H^+u-\frac{\Delta u}{2m}g.
\end{equation}
They are orthogonal, and direct differentiation gives
\begin{equation}\label{eq:ddc-hessian}
 dd^cu(X,JY)=2H^+u(X,Y),\qquad
 \Lambda_\omega dd^cu=\Delta u,\qquad
 |dd^cu|^2=2|H^+u|^2.
\end{equation}

At a point, choose a real orthonormal Ricci eigenframe
$(e_i,Je_i)_{i=1}^m$, with dual coframe $(\theta_i,\psi_i)$, and
put $\alpha_i=\theta_i\wedge\psi_i$. Then
\begin{equation}\label{eq:eigenvalues}
 \omega=\sum_i\alpha_i,\qquad
 \rho=\sum_i\lambda_i\alpha_i,\qquad
 R=2\sum_i\lambda_i,\qquad \lambda_i\geq\varepsilon R.
\end{equation}
A real $(1,1)$-form is identified with its Hermitian matrix in
the unitary coframe $(\theta_i+\mathrm i\psi_i)/\sqrt2$.
Its squared norm is the sum of the squared absolute values of the
matrix entries; in particular, $|\omega|^2=m$.

For a bounded linear map $A$ between inner-product spaces, its
\emph{operator norm} is
\[
 \norm A_{\mathrm{op}}=\sup_{\norm X=1}\norm{AX}.
\]
For a real symmetric covariant two-tensor $T$, let $T^\sharp$ be
the self-adjoint endomorphism determined by
$g(T^\sharp X,Y)=T(X,Y)$. We use the pointwise operator norm
\[
 |T|_{\mathrm{op}}
   =\norm{T^\sharp}_{\mathrm{op}}
   =\sup_{|X|=|Y|=1}|T(X,Y)|
   =\max_j|\mu_j|,
\]
where $\mu_j$ are the real eigenvalues of $T^\sharp$. Consequently,
\[
 |T(X,Y)|\leq |T|_{\mathrm{op}}|X||Y|.
\]
This differs from the tensor norm $|T|=(\sum_j\mu_j^2)^{1/2}$,
although $|T|_{\mathrm{op}}\leq |T|$.

We write $C_c^\infty(M)$ for the smooth compactly supported test
functions. The space of distributions on $M$ is defined by
\[
 \mathcal D'(M):=\bigl\{T:C_c^\infty(M)\to\mathbb C:
                 T\text{ is linear and continuous}\bigr\},
\]
where $C_c^\infty(M)$ carries its usual test-function topology.
A locally integrable function $f$ is identified with the distribution
$\phi\mapsto\int_M f\phi\,dV_g$.

We take Hilbert-space inner products to be linear in the first argument;
in $L^2(M,w\,dV_g)$, $\ip{u}{v}=\int_Mu\bar v\,w\,dV_g$.
An operator $A$ on a Hilbert space $\mathcal H$ is a linear map
\[
 A:\Dom(A)\subset\mathcal H\longrightarrow\mathcal H,
\]
where $\Dom(A)$ is a specified linear subspace of $\mathcal H$,
called the \emph{domain} of $A$: it consists of the vectors on which
$A$ is defined. The operator is \emph{densely defined} if
$\overline{\Dom(A)}=\mathcal H$, with closure taken in the norm of
$\mathcal H$. For a densely defined $A$, its \emph{adjoint} $A^*$
is defined as follows: a vector $v\in\mathcal H$ belongs to
$\Dom(A^*)$ if there exists $f\in\mathcal H$ such that
\[
 \ip{Au}{v}_{\mathcal H}=\ip{u}{f}_{\mathcal H}
 \qquad\text{for every }u\in\Dom(A).
\]
In this case $A^*v:=f$; the vector $f$ is unique because $\Dom(A)$
is dense in $\mathcal H$.

The \emph{graph} of $A$ is
\[
 \mathcal G(A):=\{(u,Au):u\in\Dom(A)\}
                   \subset\mathcal H\times\mathcal H.
\]
The operator $A$ is \emph{closed} if $\mathcal G(A)$ is closed in the
product Hilbert-space norm, and \emph{closable} if the closure of
$\mathcal G(A)$ is the graph of an operator. In the latter case,
this operator is called the \emph{closure} of $A$, denoted by
$\overline A$, and is defined by
\[
 \mathcal G(\overline A)
       =\overline{\mathcal G(A)}^{\,\mathcal H\times\mathcal H}.
\]
Explicitly, $u\in\Dom(\overline A)$ and $\overline Au=v$ if and only
if there exists a sequence $u_j\in\Dom(A)$ such that
\[
 u_j\longrightarrow u,\qquad Au_j\longrightarrow v
             \qquad\text{in }\mathcal H.
\]
Closability ensures that $v$ is uniquely determined by $u$, independently
of the approximating sequence. The closure extends $A$ and is closed.
Every densely defined symmetric operator is closable.

A densely defined operator $A$ is \emph{symmetric} if
$\ip{Au}{v}=\ip{u}{Av}$ for all $u,v\in\Dom(A)$.
It is \emph{self-adjoint} if $A=A^*$, including
equality of domains. A densely defined symmetric operator $A_0$ is
\emph{essentially self-adjoint} if its closure $\overline{A_0}$ is
self-adjoint, equivalently $\overline{A_0}=A_0^*$. In this case
$\overline{A_0}$ is its unique self-adjoint extension.

For a closed operator $A$, its \emph{graph norm} is
\[
 \norm{u}_{A}=\bigl(\norm{u}_{\mathcal H}^2
                         +\norm{Au}_{\mathcal H}^2\bigr)^{1/2},
 \qquad u\in\Dom(A).
\]
A linear subspace $\mathcal C\subset\Dom(A)$ is a \emph{graph core}
for $A$ if it is dense in $\Dom(A)$ in this norm. Equivalently, for
every $u\in\Dom(A)$ there are $u_j\in\mathcal C$ such that
\[
 \norm{u_j-u}_{\mathcal H}+\norm{Au_j-Au}_{\mathcal H}
 \longrightarrow0.
\]
These operator notions refer to the stated Hilbert space.
In Lemma~\ref{lem:operator}, the space is $L^2(M,w\,dV_g)$.

\subsection*{Spectral measures and functional calculus}
Let $S$ be a nonnegative self-adjoint operator on a complex Hilbert
space $\mathcal H$. The spectral theorem associates to $S$ a
\emph{projection-valued spectral measure} $E_S$ on $[0,\infty)$;
see, for example, \cite[Section~4.7]{Gross}.
For each Borel set $B$, $E_S(B)$ is an orthogonal projection on
$\mathcal H$, selecting the part of a vector with spectral values
in $B$. It satisfies
\[
 E_S(\varnothing)=0,\qquad E_S([0,\infty))=I,\qquad
 E_S(B)E_S(C)=E_S(B\cap C).
\]
For pairwise disjoint Borel sets $B_j$,
$E_S(\bigcup_jB_j)F=\sum_jE_S(B_j)F$, with convergence in
$\mathcal H$ for every $F\in\mathcal H$.
Thus $E_S$ takes values in operators. For a fixed vector $F$, the
associated \emph{scalar spectral measure}
\[
 \nu_F(B):=\norm{E_S(B)F}_{\mathcal H}^2
          =\ip{E_S(B)F}{F}_{\mathcal H},\qquad
 \nu_F([0,\infty))=\norm F_{\mathcal H}^2,
\]
is an ordinary finite nonnegative measure.

For a bounded Borel function $b$, the functional calculus defines
\[
 \begin{aligned}
 b(S)&=\int_{[0,\infty)}b(\lambda)\,dE_S(\lambda),\\
 \norm{b(S)F}_{\mathcal H}^2
   &=\int_{[0,\infty)}|b(\lambda)|^2\,d\nu_F(\lambda),\\
 \ip{b(S)F}{F}_{\mathcal H}
   &=\int_{[0,\infty)}b(\lambda)\,d\nu_F(\lambda).
 \end{aligned}
\]
The same construction gives
\[
 S=\int_{[0,\infty)}\lambda\,dE_S(\lambda),\qquad
 \Dom(S)=\left\{F\in\mathcal H:
       \int_{[0,\infty)}\lambda^2\,d\nu_F(\lambda)<\infty\right\}.
\]
For a Hermitian matrix with an orthonormal eigenbasis $e_j$,
$Se_j=\lambda_je_j$ and $F=\sum_jc_je_j$, these formulas reduce to
$E_S(B)F=\sum_{\lambda_j\in B}c_je_j$ and
$\nu_F=\sum_j|c_j|^2\delta_{\lambda_j}$, with
$\delta_{\lambda_j}$ the unit point mass at $\lambda_j$.
The general formulas require no eigenbasis. The orthogonal
projection $P_{\ker S}$ onto $\ker S$ satisfies
\[
 E_S(\{0\})=P_{\ker S},\qquad
 \nu_F(\{0\})=\norm{P_{\ker S}F}_{\mathcal H}^2.
\]
\subsection*{Volume comparison and cutoffs}
Write $V(r)=\operatorname{Vol}B(o,r)$ for a fixed $o\in M$.
Since $\Ric\geq0$, Bishop--Gromov volume comparison gives
\begin{equation}\label{eq:volume}
 V(r)\leq\omega_{2m}r^{2m},
\end{equation}
where $\omega_{2m}$ is the Euclidean unit-ball volume.
Moreover, $M$ has infinite volume. Indeed, let $\gamma$ be a
minimizing ray from $o$ and set $r_j=4^j$. The balls
$B(\gamma(r_j),r_j/2)$ are pairwise disjoint, and comparison with
$B(\gamma(r_j),r_j+1)\supset B(o,1)$ shows that each has volume at
least $4^{-2m}V(1)$.

Only gradient bounds for cutoffs will be used. For $0<s<t$, define
\begin{equation}\label{eq:cutoff}
 \eta_{s,t}(x)=
 \begin{cases}
  1, & d(o,x)\leq s,\\[2pt]
  \dfrac{t-d(o,x)}{t-s}, & s<d(o,x)<t,\\[4pt]
  0, & d(o,x)\geq t.
 \end{cases}
\end{equation}
It has compact support by completeness, equals one on $B(o,s)$,
and satisfies $|\nabla\eta_{s,t}|\leq(t-s)^{-1}$ almost everywhere.
The compactly supported Lipschitz cutoffs used here are admissible
in the estimates below, which involve only the test function and its
first weak derivatives with locally bounded coefficients.
Proposition~\ref{prop:compact-approximation} in Appendix~\ref{app:approximation}
provides smooth approximations with a common compact support and
strong $W^{1,2}$ convergence, and proves passage to all the weighted
integrals used here. It also justifies the compactly supported
Sobolev tests in Lemma~\ref{lem:operator}. This approximation is
distinct from approximation in an operator graph norm.

\section{Resolvents and quadratic Ricci pinching}\label{sec:resolvent-estimates}

In this section, we derive a localized quadratic estimate from resolvents and
K\"ahler energy identities. This yields the surface case of
Theorem~\ref{thm:main}.

\subsection*{Resolvents and the K\"ahler energy identities}

We begin with the operator facts used for both the ordinary
Laplacian and the mixed operators.

\begin{lemma}\label{lem:operator}
Let $w$ be a smooth positive function on a complete Riemannian
manifold, and let $D$ be a smooth positive definite symmetric tensor
such that $D\leq Cwg$. The operator
\[
 S_0u=-w^{-1}\Div(D\nabla u),\qquad
 \Dom(S_0)=C_c^\infty(M),
\]
is essentially self-adjoint and nonnegative in $L^2(M,w\,dV_g)$.
Its closure $S$ has $C_c^\infty(M)$ as a graph core, and its domain is
\[
 \begin{gathered}
 \Dom(S)=\bigl\{u\in L^2(M,w\,dV_g):\\
 -w^{-1}\Div(D\nabla u)\in L^2(M,w\,dV_g)
 \text{ in }\mathcal D'(M)\bigr\}.
 \end{gathered}
\]
Here the differential expression is interpreted as a local
distribution relative to the Riemannian density $dV_g$:
$Su=f$ means
\[
 \int_M u\,[-\Div(D\nabla\overline\phi)]\dd
       =\int_M wf\overline\phi\dd,
 \qquad \phi\in C_c^\infty(M).
\]
In particular, the distribution is represented by a function in the
stated weighted $L^2$ space. The closure satisfies
\begin{equation}\label{eq:energy}
 \int_M D(\nabla u,\nabla\bar u)\dd
      =\ip{u}{Su}_{L^2(w\,dV_g)},\qquad u\in\Dom(S).
\end{equation}
For $F\in L^2(w\,dV_g)$ and $\tau>0$, put
$u=(S+\tau)^{-1}F$ and $v=Su$. Then
\begin{equation}\label{eq:spectral}
 \begin{gathered}
 \norm v_{L^2(w\,dV_g)}\leq\norm F_{L^2(w\,dV_g)},\qquad
 \tau\norm u_{L^2(w\,dV_g)}\leq\norm F_{L^2(w\,dV_g)},\\
 \ip{u}{Su}\leq\frac{\norm F_{L^2(w\,dV_g)}^2}{4\tau}.
 \end{gathered}
\end{equation}
If $\ker S=\{0\}$, then $v\to F$ strongly as $\tau\to0$.
\end{lemma}

\begin{proof}
Put $\mathcal H=L^2(M,w\,dV_g)$. Since $w$ is smooth and positive,
$C_c^\infty(M)$ is dense in $\mathcal H$. Integration by parts gives
\[
 \ip{S_0f}{g}_{\mathcal H}
   =\int_M D(\nabla f,\nabla\bar g)\dd
   =\ip{f}{S_0g}_{\mathcal H},\qquad f,g\in C_c^\infty(M).
\]
Thus $S_0$ is symmetric and nonnegative, and is closable.
The definition of the adjoint shows that
$u\in\Dom(S_0^*)$ and $S_0^*u=f$ are equivalent to the testing
identity in the statement. Indeed, that identity is precisely
$\ip{u}{S_0\phi}_{\mathcal H}=\ip{f}{\phi}_{\mathcal H}$.
All distributions are well defined locally, since $w$ and $w^{-1}$
are bounded on compact sets. This identifies the adjoint domain
with the asserted maximal distributional domain.

\smallskip
\noindent\emph{Essential self-adjointness.}
We first prove $\ker(S_0^*+1)=\{0\}$. If $z$ belongs to this
kernel, the preceding characterization gives
\[
 -\Div(D\nabla z)+wz=0\qquad\text{in }\mathcal D'(M).
\]
The coefficients are smooth and the equation is uniformly elliptic
on each compact subset, so local elliptic regularity gives
$z\in C^\infty(M)$. On each connected component, take
$\chi_r=\eta_{r,2r}$ from \eqref{eq:cutoff}.
Set $\phi=\chi_r^2z$, by Proposition~\ref{prop:compact-approximation}, $\phi$ is an
admissible compactly supported $W^{1,2}$ test function for the
preceding distributional equation $-\Div(D\nabla z)+wz=0$.
Its weak formulation gives
\begin{align*}
 0&=\int_M D(\nabla z,\nabla\bar\phi)\dd
      +\int_M wz\bar\phi\dd\\
 &=\int_M\chi_r^2D(\nabla z,\nabla\bar z)\dd
       +\int_M\chi_r^2w|z|^2\dd\\
  &\quad+2\int_M\chi_r\bar z\,D(\nabla z,\nabla\chi_r)\dd.
\end{align*}
The first two integrals are real, since $D$ is real symmetric
and $w$ and $\chi_r$ are real valued. Moving the cross term to the right yields
\begin{align*}
 &\int_M\chi_r^2D(\nabla z,\nabla\bar z)\dd
     +\int_M\chi_r^2w|z|^2\dd\\
 &\qquad=-2\int_M\chi_r\bar z\,D(\nabla z,\nabla\chi_r)\dd.
\end{align*}
To bound the absolute value of the cross term, Cauchy--Schwarz
for the positive tensor $D$ and $2ab\leq\tfrac12a^2+2b^2$ give
\begin{align*}
 &2\chi_r|z|\,|D(\nabla z,\nabla\chi_r)|\\
 &\quad\leq
 2\bigl(\chi_r D(\nabla z,\nabla\bar z)^{1/2}\bigr)
   \bigl(|z|D(\nabla\chi_r,\nabla\chi_r)^{1/2}\bigr)\\
 &\quad\leq\frac12\chi_r^2D(\nabla z,\nabla\bar z)
       +2|z|^2D(\nabla\chi_r,\nabla\chi_r).
\end{align*}
Finally, $D\leq Cwg$ and \eqref{eq:cutoff} imply
$D(\nabla\chi_r,\nabla\chi_r)\leq Cw|\nabla\chi_r|^2
\leq Cw/r^2$ almost everywhere. Integrating the pointwise
estimate and absorbing half of the gradient integral gives
\begin{align*}
 &\int_M\chi_r^2w|z|^2\dd
    +\frac12\int_M\chi_r^2D(\nabla z,\nabla\bar z)\dd\\
 &\qquad\leq2\int_M|z|^2D(\nabla\chi_r,\nabla\chi_r)\dd
 \leq\frac{2C}{r^2}\norm z_{\mathcal H}^2.
\end{align*}
Since $\chi_r=1$ on $B(o,r)$, we have 
\[0\leq\int_{B(o,r)}|z|^2w\dd\leq\frac{2C}{r^2}\norm z_{\mathcal H}^2.\]
letting $r\to\infty$ proves
$z=0$ on each component.

Now we show that $S_0$ is essentially self-adjoint. Let $A=\overline{S_0}$. Nonnegativity gives
$\norm{(A+1)f}_{\mathcal H}\geq\norm f_{\mathcal H}$ for
$f\in\Dom(A)$. As $A$ is closed, this inequality makes
$\operatorname{Ran}(A+1)$ closed: if $(A+1)f_j$ converges, then
$f_j$ and $Af_j$ are Cauchy, so their limits belong to the graph of
$A$. Moreover,
\[
 \operatorname{Ran}(A+1)^\perp
   =\ker(A^*+1)=\ker(S_0^*+1)=\{0\}.
\]
The range is therefore all of $\mathcal H$. Since $(A+1)$ is surjective, given
$y\in\Dom(A^*)$, we can choose $x\in\Dom(A)$ with
$(A+1)x=(A^*+1)y$. Then $y-x\in\ker(A^*+1)$, so $y=x\in\Dom(A)$.
Thus $A=A^*$ and $S_0$ is essentially self-adjoint. Consequently,
$S=A=S_0^*$ has the stated domain, and $C_c^\infty(M)$ is a graph
core by the definition of the closure.

\smallskip
\noindent\emph{Energy identity on the full domain.}
Fix $u\in\Dom(S)$. Since $C_c^\infty(M)$ is a graph core for
$S$, there are $u_j\in C_c^\infty(M)$ such that
\[
 \norm{u_j-u}_{\mathcal H}\longrightarrow0,
 \qquad
 \norm{S_0u_j-Su}_{\mathcal H}\longrightarrow0.
\]
Using $g$ to identify $D$ with its positive self-adjoint
endomorphism, let $D^{1/2}$ denote its positive square root.
For $d_{j\ell}=u_j-u_\ell$, the compactly supported
integration-by-parts identity gives
\begin{align*}
 \norm{D^{1/2}\nabla d_{j\ell}}_{L^2(dV_g)}^2
 &=\int_M D(\nabla d_{j\ell},\nabla\bar d_{j\ell})\dd\\
 &=\ip{d_{j\ell}}{S_0d_{j\ell}}_{\mathcal H}\\
 &\leq\norm{d_{j\ell}}_{\mathcal H}
          \norm{S_0d_{j\ell}}_{\mathcal H}
 \longrightarrow0.
\end{align*}
Both factors tend to zero by the two graph-norm convergences
above. Thus $D^{1/2}\nabla u_j$ is Cauchy in the space of
square-integrable complex vector fields with measure $dV_g$.
Its completeness gives a vector field $V\in L^2(dV_g)$ such that
\[
 D^{1/2}\nabla u_j\longrightarrow V
 \qquad\text{strongly in }L^2(dV_g).
\]

We identify $V$ by local convergence. Fix a compact set
$K\subset M$. Smoothness and positivity give constants
$\alpha_K,\beta_K>0$ such that
\[
 w\geq\alpha_K,\qquad D\geq\beta_Kg
 \quad\text{on }K.
\]
Set $G=D^{-1/2}V$, which belongs to $L^2_{\mathrm{loc}}$.
The preceding bounds imply
\begin{align*}
 \norm{u_j-u}_{L^2(K,dV_g)}^2
 &\leq\alpha_K^{-1}\norm{u_j-u}_{\mathcal H}^2
 \longrightarrow0,\\
 \norm{\nabla u_j-G}_{L^2(K,dV_g)}^2
 &\leq\beta_K^{-1}
       \norm{D^{1/2}\nabla u_j-V}_{L^2(K,dV_g)}^2
 \longrightarrow0.
\end{align*}
For any real smooth compactly supported vector field $X$,
these local convergences yield
\begin{align*}
 \int_M g(G,X)\dd
 &=\lim_{j\to\infty}\int_M g(\nabla u_j,X)\dd\\
 &=-\lim_{j\to\infty}\int_M u_j\Div X\dd
  =-\int_Mu\Div X\dd.
\end{align*}
This is the distributional identity defining the weak gradient.
Hence $u\in W^{1,2}_{\mathrm{loc}}(M)$ and $G=\nabla u$.
Consequently $V=D^{1/2}\nabla u$, so the global convergence is
\[
 D^{1/2}\nabla u_j\longrightarrow D^{1/2}\nabla u
 \quad\text{in }L^2(dV_g).
\]
In particular, the weighted gradient energy of $u$ is finite.

It remains to pass to the limit on the operator side. The
graph-norm convergences give
\begin{align*}
 &\bigl|\ip{u_j}{S_0u_j}_{\mathcal H}
            -\ip{u}{Su}_{\mathcal H}\bigr|\\
 &\quad\leq\norm{u_j-u}_{\mathcal H}
              \norm{S_0u_j}_{\mathcal H}
       +\norm u_{\mathcal H}
              \norm{S_0u_j-Su}_{\mathcal H}
 \longrightarrow0,
\end{align*}
since $S_0u_j\to Su$ also makes $\norm{S_0u_j}_{\mathcal H}$
bounded. Taking limits in the compactly supported identities
therefore gives
\begin{align*}
 \int_M D(\nabla u,\nabla\bar u)\dd
 &=\norm{D^{1/2}\nabla u}_{L^2(dV_g)}^2\\
 &=\lim_{j\to\infty}
          \norm{D^{1/2}\nabla u_j}_{L^2(dV_g)}^2\\
 &=\lim_{j\to\infty}\ip{u_j}{S_0u_j}_{\mathcal H}
  =\ip{u}{Su}_{\mathcal H},
\end{align*}
which proves \eqref{eq:energy} on all of $\Dom(S)$.

\smallskip
\noindent\emph{Resolvent bounds.}
Since $S$ is nonnegative and self-adjoint, its spectrum is
contained in $[0,\infty)$, and $(S+\tau I)^{-1}$ is defined
on all of $\mathcal H$ for every $\tau>0$.
By the definitions of $u$ and $v$,
\[
 u=(S+\tau I)^{-1}F,\qquad v=Su,\qquad F=v+\tau u.
\]
We now use the spectral theorem to estimate these vectors.
Let $E_S$ and $\nu_F$ be the spectral measure and scalar spectral
measure from Section~\ref{sec:preliminaries}, so that
\[
 \nu_F(B)=\norm{E_S(B)F}_{\mathcal H}^2,\qquad
 \nu_F([0,\infty))=\norm F_{\mathcal H}^2.
\]
The functional-calculus formulas are
\begin{align*}
 u&=\int_{[0,\infty)}\frac{1}{\lambda+\tau}\,dE_S(\lambda)F,\\
 v&=\int_{[0,\infty)}\frac{\lambda}{\lambda+\tau}
                                      \,dE_S(\lambda)F,\\
 \tau u&=\int_{[0,\infty)}\frac{\tau}{\lambda+\tau}
                                      \,dE_S(\lambda)F.
\end{align*}
Here a \emph{spectral multiplier} is the scalar Borel function
$b(\lambda)$ used in the representation
$b(S)=\int b(\lambda)\,dE_S(\lambda)$; the three multipliers
above are explicitly $1/(\lambda+\tau)$,
$\lambda/(\lambda+\tau)$, and $\tau/(\lambda+\tau)$.

For $\lambda\geq0$ and $\tau>0$,
\[
 0\leq\frac{\lambda}{\lambda+\tau}\leq1,
 \qquad
 0\leq\frac{\tau}{\lambda+\tau}\leq1.
\]
Applying the norm formula for spectral integrals gives
\begin{align*}
 \norm v_{\mathcal H}^2
 &=\int_{[0,\infty)}
       \left(\frac{\lambda}{\lambda+\tau}\right)^2
                                      \,d\nu_F(\lambda)\\
 &\leq\int_{[0,\infty)}1\,d\nu_F(\lambda)
   =\norm F_{\mathcal H}^2,\\
 \tau^2\norm u_{\mathcal H}^2
 &=\int_{[0,\infty)}
       \left(\frac{\tau}{\lambda+\tau}\right)^2
                                      \,d\nu_F(\lambda)\\
 &\leq\int_{[0,\infty)}1\,d\nu_F(\lambda)
   =\norm F_{\mathcal H}^2.
\end{align*}
Taking square roots proves the first two bounds in
\eqref{eq:spectral}.

For the energy, the same functional calculus gives
\begin{align*}
 \ip{u}{Su}_{\mathcal H}
 &=\int_{[0,\infty)}
       \frac{1}{\lambda+\tau}\frac{\lambda}{\lambda+\tau}
                                      \,d\nu_F(\lambda)\\
 &=\int_{[0,\infty)}\frac{\lambda}{(\lambda+\tau)^2}
                                      \,d\nu_F(\lambda).
\end{align*}
The scalar estimate follows from
\[ 0\leq\frac{\lambda}{(\lambda+\tau)^2}\leq\frac1{4\tau}.
\]
Consequently,
\[
 \ip{u}{Su}_{\mathcal H}
 \leq\frac1{4\tau}\int_{[0,\infty)}1\,d\nu_F(\lambda)
 =\frac{\norm F_{\mathcal H}^2}{4\tau},
\]
which is the third bound in \eqref{eq:spectral}.
The weighted gradient energy satisfies the same bound by
\eqref{eq:energy}.

To prove the strong convergence, use $F=v+\tau u$ to write
$v-F=-\tau u$.
Hence the same norm formula yields
\[
 \norm{v-F}_{\mathcal H}^2
 =\int_{[0,\infty)}
       \left(\frac{\tau}{\lambda+\tau}\right)^2
                                      \,d\nu_F(\lambda).
\]
As $\tau\to0$, the integrand satisfies
\[
 \left(\frac{\tau}{\lambda+\tau}\right)^2
 \longrightarrow
 \begin{cases}
  1,&\lambda=0,\\
  0,&\lambda>0,
 \end{cases}
 \qquad
 0\leq\left(\frac{\tau}{\lambda+\tau}\right)^2\leq1.
\]
Since $\nu_F$ is finite, dominated convergence therefore gives
\begin{align*}
 \lim_{\tau\downarrow0}\norm{v-F}_{\mathcal H}^2
 &=\nu_F(\{0\})\\
 &=\norm{E_S(\{0\})F}_{\mathcal H}^2
  =\norm{P_{\ker S}F}_{\mathcal H}^2.
\end{align*}
If $\ker S=\{0\}$, this limit is zero, which proves
$v\to F$ strongly in $\mathcal H$.
Finally, for smooth $F$, local elliptic regularity applied to
$-\Div(D\nabla u)+\tau wu=wF$ gives $u\in C^\infty(M)$.
This local assertion does not itself supply a global weighted
derivative estimate.
\end{proof}

Apply the lemma with $w=1$ and $D=g$ to obtain the nonnegative
self-adjoint realization of $L=-\Delta$. Its kernel is zero here:
\eqref{eq:energy} makes any kernel element constant, and a nonzero
constant is not square integrable on our connected infinite-volume
manifold. For general elements of an operator domain, derivatives
such as $\nabla^2u$ and $dd^cu$ are initially distributional.
In the identities below they denote the $L^2$ representatives
obtained by graph-core approximation, in the stated measure and
weight. Smoothness is asserted separately for resolvents with smooth
sources, by the local regularity argument in
Lemma~\ref{lem:operator}.

The following identities give the
K\"ahler improvement over the ordinary Hessian estimate.

\begin{lemma}\label{lem:kahler-energy}
For every real $u\in\Dom(L)$,
\begin{align}
 \int_M|H^+u|^2\dd&=\frac12\norm{Lu}_2^2,
 \int_M|H^+_0u|^2\dd=c_m\norm{Lu}_2^2,\label{eq:kahler-energy}\\
 &\int_M|H^-u|^2\dd+\int_M\Ric(\nabla u,\nabla u)\dd
     =\frac12\norm{Lu}_2^2.\label{eq:ricci-energy}
\end{align}
\end{lemma}

\begin{proof}
For $u\in C_c^\infty(M)$, put $A=dd^cu$. The identities
$dA=d^2(d^cu)=0$ and $d\omega=0$ show that
$A\wedge\omega^{m-2}$ is closed. Hence
\[
 \frac{A\wedge A\wedge\omega^{m-2}}{(m-2)!}
   =d\left(\frac{d^cu\wedge A\wedge\omega^{m-2}}{(m-2)!}\right).
\]
The form inside the exterior derivative has compact support,
so Stokes' theorem makes the integral zero.
Since $u$ is real, $A$ is a real
$(1,1)$-form. Choose a real orthonormal frame
$(e_i,Je_i)_{i=1}^m$ in which $A$ is diagonal. If
$(\theta_i,\psi_i)$ is the dual coframe and
$\alpha_i=\theta_i\wedge\psi_i$, then
\[
 A=\sum_{i=1}^m a_i\alpha_i,\qquad
 \omega=\sum_{i=1}^m\alpha_i,\qquad
 dV_g=\alpha_1\wedge\cdots\wedge\alpha_m.
\]
The forms $\alpha_i$ satisfy
$\alpha_i\wedge\alpha_i=0$ and
$\alpha_i\wedge\alpha_j=\alpha_j\wedge\alpha_i$. Hence
\[
 A\wedge A=2\sum_{i<j}a_ia_j\,\alpha_i\wedge\alpha_j,
 \qquad
 \frac{\omega^{m-2}}{(m-2)!}
 =\sum_{\substack{I\subset\{1,\ldots,m\}\\|I|=m-2}}\alpha_I,
\]
where $\alpha_I=\bigwedge_{i\in I}\alpha_i$, with the empty
product interpreted as $1$. For each pair $i<j$, only the
complementary index set
$I=\{1,\ldots,m\}\setminus\{i,j\}$ contributes to the wedge
product. Therefore
\[
\begin{aligned}
 \frac{A\wedge A\wedge\omega^{m-2}}{(m-2)!}
 &=2\sum_{i<j}a_ia_j\,dV_g=\left[
       \left(\sum_{i=1}^m a_i\right)^2
       -\sum_{i=1}^m a_i^2
   \right]dV_g.
\end{aligned}
\]

The forms $\alpha_i$ are orthonormal, so
$|A|^2=\sum_i a_i^2$. Moreover, the identity
$A(X,JY)=2H^+u(X,Y)$ shows that the eigenvalues of $H^+u$
are $a_i/2$, each with multiplicity two. Consequently,
\[
\begin{aligned}
 \Delta u
 &=\operatorname{tr}_g H^+u
   =2\sum_{i=1}^m\frac{a_i}{2}
   =\sum_{i=1}^m a_i,\\
 |H^+u|^2
 &=2\sum_{i=1}^m\left(\frac{a_i}{2}\right)^2
   =\frac12\sum_{i=1}^m a_i^2
   =\frac12|dd^cu|^2.
\end{aligned}
\]
Integrating therefore gives
$\int_M|H^+u|^2\dd=\frac12\norm{Lu}_2^2$.
Since
$|H^+_0u|^2=|H^+u|^2-(\Delta u)^2/(2m)$,
subtracting the trace part gives \eqref{eq:kahler-energy}.
The integrated Bochner formula reads
\[
 \int_M|\nabla^2u|^2\dd+
 \int_M\Ric(\nabla u,\nabla u)\dd=\norm{Lu}_2^2.
\]
Orthogonality of $H^+u$ and $H^-u$ now gives
\eqref{eq:ricci-energy}.

For general $u\in\Dom(L)$, choose real
$u_j\in C_c^\infty(M)$ with $u_j\to u$ and $Lu_j\to Lu$
in $L^2(dV_g)$. Apply the identities and Bochner's formula to
$u_j-u_\ell$. Since $\Ric\geq0$,
\[
 \|\nabla^2(u_j-u_\ell)\|_2^2
 +\int_M\Ric(\nabla(u_j-u_\ell),\nabla(u_j-u_\ell))\dd
 =\|L(u_j-u_\ell)\|_2^2\longrightarrow0.
\]
The Hessians therefore converge in $L^2(dV_g)$, and their limit
is $\nabla^2u$ by local distributional convergence. The gradients
converge in $L^2(dV_g)$ by \eqref{eq:energy}; the Ricci-weighted
limit is identified locally with the same $\nabla u$, since the
smooth nonnegative Ricci tensor is bounded on compact sets.
Thus all displayed identities pass to the limit with weak
Hessians, without any curvature upper bound. For the smooth
sources used later, local elliptic regularity additionally makes
the corresponding resolvent solutions smooth.
\end{proof}

\subsection*{The localized quadratic estimate}

The first curvature estimate uses only the ordinary Laplacian.
It is a localization of the contracted-Bianchi and Hessian pairing
that also underlies almost-Schur inequalities; compare
\cite{DeLellisTopping}. The compactly supported source is important
because no global curvature integral has yet been established.

\begin{proposition}\label{prop:quadratic}
Let $(M^{2m},g,J)$ be the K\"ahler manifold in Theorem \ref{thm:main}.
Then, for every real $\eta\in C_c^\infty(M)$,
\begin{equation}\label{eq:quadratic}
 \int_M\eta^2R^2\dd
 \leq\frac{(1/2-\varepsilon)^2}
               {2(m-1)^2\varepsilon^3}
               \int_M R|\nabla\eta|^2\dd.
\end{equation}
\end{proposition}

\begin{proof}
Set $E_m=\Ric-Rg/(2m)$ and $T=\Ric-Rg/2$.
The contracted Bianchi identity $\Div\Ric=\frac12\,dR$ gives
$\Div E_m=c_m\,dR$. Applying the product rule, and using
$E_m-c_mRg=T$, we obtain
\begin{equation}\label{eq:bianchi-source}
 \Div(\eta E_m)=c_m\,d(\eta R)+T(\nabla\eta,\cdot).
\end{equation}
We first record the pointwise bounds
\begin{equation}\label{eq:ordinary-algebra}
 |E_m|^2\leq c_mq^2R^2,\qquad
 |T|_{\mathrm{op}}\leq(1/2-\varepsilon)R.
\end{equation}
Indeed, put $\mu_i=\lambda_i-\varepsilon R$ in the paired
Ricci eigenframe \eqref{eq:eigenvalues}. Then
$\mu_i\geq0$ and $\sum_i\mu_i=qR/2$, so
$\sum_i\mu_i^2\leq(\sum_i\mu_i)^2$. Since each Ricci eigenvalue
occurs twice in the real frame,
\[
 |E_m|^2=2\left(\sum_i\mu_i^2
                      -\frac{q^2R^2}{4m}\right)
 \leq\frac{m-1}{2m}q^2R^2.
\]
The eigenvalues of $T^\sharp$ are $\lambda_i-R/2$, again
with multiplicity two. The bounds
$\varepsilon R\leq\lambda_i\leq R/2$ place them in
$[-(1/2-\varepsilon)R,0]$, proving the operator-norm bound.

Set
\[
 \begin{gathered}
 F=\eta R,\qquad X=\norm F_2,\qquad
 G=\left(\int_M R|\nabla\eta|^2\dd\right)^{1/2},\\
 u_\tau=(L+\tau)^{-1}F,\qquad v_\tau=Lu_\tau.
 \end{gathered}
\]
The source $F$ is real, smooth, and compactly supported, so
$F\in L^2(M)$ without any global integrability assumption on $R$.
The resolvent solution is real and smooth, and
Lemma~\ref{lem:operator} gives $\norm{v_\tau}_2\leq X$.
Pairing \eqref{eq:bianchi-source} with $\nabla u_\tau$
\[
 \int_M
 \bigl\langle\Div(\eta E_m),\nabla u_\tau\bigr\rangle\,dV_g
 =
 c_m\int_M\langle\nabla F,\nabla u_\tau\rangle\,dV_g
 +\int_M T(\nabla\eta,\nabla u_\tau)\,dV_g.
\]
Integrating by parts gives
\[
 c_m\ip F{v_\tau}_{L^2(\dd)}
 =-\int_M\eta\,\ip{E_m}{\nabla^2 u_\tau}\dd
  -\int_M T(\nabla\eta,\nabla u_\tau)\dd.
\]
The K\"ahler condition implies
$E_m(JX,JY)=E_m(X,Y)$, and the definition of $E_m$ gives
$\operatorname{tr}_gE_m=0$. At any point, choose a real
orthonormal frame $(e_a)_{a=1}^{2m}$. Using the $J$-invariance
of $E_m$, we obtain
\[
 \begin{aligned}
 \langle E_m,\nabla^2u_\tau\rangle
 &=
 \frac12\sum_{a,b=1}^{2m}E_m(e_a,e_b)
 \bigl(
   \nabla^2u_\tau(e_a,e_b)
   +\nabla^2u_\tau(Je_a,Je_b)
 \bigr)\\
 &=\langle E_m,H^+u_\tau\rangle\\
 &=\langle E_m,H^+_0u_\tau\rangle
   +\frac{\Delta u_\tau}{2m}\operatorname{tr}_gE_m\\
 &=\langle E_m,H^+_0u_\tau\rangle.
 \end{aligned}
\]
Substituting this into the preceding identity gives
\[
 c_m\langle F,v_\tau\rangle_{L^2(\dd)}
 =
 -\int_M\eta\,\langle E_m,H^+_0u_\tau\rangle\,dV_g
 -\int_M T(\nabla\eta,\nabla u_\tau)\,dV_g.
\]
The integration by parts involves
compactly supported tensors; the global Hessian and energy norms
used below are supplied by Lemma~\ref{lem:kahler-energy}.

For the first term, Cauchy--Schwarz, \eqref{eq:ordinary-algebra}, \eqref{eq:kahler-energy} and \eqref{eq:spectral} yield
\[
 \begin{aligned}
 \left|\int_M\eta\,\ip{E_m}{H^+_0u_\tau}\dd\right|
 &\leq \sqrt{c_m}\,qX\,\norm{H^+_0u_\tau}_2\\
 &=c_mqX\norm{v_\tau}_2
 \leq c_mqX^2.
 \end{aligned}
\]
For the second term, the operator-norm bound gives
$|T(\nabla\eta,\nabla u_\tau)|
\leq(1/2-\varepsilon)R|\nabla\eta||\nabla u_\tau|$.
Moreover, pinching and\eqref{eq:spectral}, \eqref{eq:ricci-energy} imply
\[
 \varepsilon\int_M R|\nabla u_\tau|^2\dd
 \leq\int_M\Ric(\nabla u_\tau,\nabla u_\tau)\dd
 \leq\frac12\norm{v_\tau}_2^2
 \leq\frac12X^2.
\]
Thus weighted Cauchy--Schwarz gives the bound, uniform in $\tau$,
\begin{align*}
 \left|\int_M T(\nabla\eta,\nabla u_\tau)\dd\right|
 &\leq(1/2-\varepsilon)G
       \left(\int_M R|\nabla u_\tau|^2\dd\right)^{1/2}\\
 &\leq\frac{1/2-\varepsilon}{\sqrt{2\varepsilon}}
       XG.
\end{align*}
Combining the two estimates, we have
\[
 c_m\ip F{v_\tau}_{L^2(\dd)}\leq c_mqX^2
       +\frac{1/2-\varepsilon}{\sqrt{2\varepsilon}}XG.
\]
Since $\ker L=\{0\}$ on our connected infinite-volume manifold,
Lemma~\ref{lem:operator} gives $v_\tau\to F$ strongly in $L^2$.
Letting $\tau\downarrow0$ therefore yields
\[
 c_m(1-q)X^2
 \leq\frac{1/2-\varepsilon}{\sqrt{2\varepsilon}}XG.
\]
If $X=0$, the conclusion is immediate. Otherwise divide by $X$
and use $c_m(1-q)=(m-1)\varepsilon>0$ to obtain
\[
 X\leq
 \frac{1/2-\varepsilon}{\sqrt{2}\,(m-1)\varepsilon^{3/2}}G.
\]
Squaring is exactly \eqref{eq:quadratic}.
\end{proof}

The quadratic estimate already can settles the critical surface case.
We give the ball iteration and the final annular argument explicitly.

\begin{corollary}[Ricci pinching on K\"ahler surfaces]\label{cor:kahler-surface}
Let $(M^4,g,J)$ be a connected complete K\"ahler surface satisfying
\eqref{eq:pinching} for a fixed $\varepsilon>0$. Then $M$ is compact
or Ricci-flat.
\end{corollary}

\begin{proof}
Assume that $M$ is noncompact. Tracing \eqref{eq:pinching}
settles $\varepsilon>1/4$, so we may suppose
$0<\varepsilon\leq1/4$. Put
$I_2(s)=\int_{B(o,s)}R^2\dd$ and $V(s)=\operatorname{Vol}B(o,s)$.
These integrals are finite for finite $s$, by completeness and
smoothness of $R$.

For $0<s<t$, apply Proposition~\ref{prop:quadratic} to
$\eta_{s,t}$ from \eqref{eq:cutoff}. Its admissibility follows
from Proposition~\ref{prop:compact-approximation}.
The cutoff properties and Cauchy--Schwarz give
\[
 I_2(s)\leq\frac{C}{(t-s)^2}\int_{B(o,t)}R\dd
       \leq\frac{C}{(t-s)^2}I_2(t)^{1/2}V(t)^{1/2}.
\]
Young's inequality then gives, for every $\theta>0$,
\[
 I_2(s)\leq\theta I_2(t)+C_\theta(t-s)^{-4}V(t).
\]
Fix $r>0$, choose $0<\theta<1/16$, and set
$s_j=2r-r2^{-j}$ for $j\geq0$. Then $s_0=r$, $s_j<2r$,
and $s_{j+1}-s_j=r2^{-j-1}$. Repeated substitution, with the
fixed factor $16$ absorbed into $C_\theta$, gives
\[
 I_2(r)\leq\theta^NI_2(2r)
       +C_\theta r^{-4}V(2r)
           \sum_{j=0}^{N-1}(16\theta)^j.
\]
For this fixed $r$, the outer-ball integral $I_2(2r)$ is finite.
Hence the first term tends to zero as $N\to\infty$, and the
series converges because $16\theta<1$. Using
\eqref{eq:volume} in real dimension four, we obtain
\begin{equation}\label{eq:surface-ball-estimate}
 I_2(r)\leq Cr^{-4}V(2r)\leq C.
\end{equation}
The constant is independent of $r$. Monotone convergence as
$r\to\infty$ now proves $\int_MR^2\dd<\infty$.
Thus global integrability is a consequence of the local estimate.

Apply Proposition~\ref{prop:quadratic} once more with
$\eta=\eta_{r,2r}$, keeping the gradient supported in the annulus
$A_r=\overline{B(o,2r)}\setminus B(o,r)$. Since
$|\nabla\eta|\leq r^{-1}$ and $V(2r)\leq Cr^4$,
\begin{equation}\label{eq:surface-annular-estimate}
 \begin{aligned}
 I_2(r)&\leq Cr^{-2}\int_{A_r}R\dd\\
 &\leq Cr^{-2}
       \left(\int_{A_r}R^2\dd\right)^{1/2}V(2r)^{1/2}
 \leq C\left(\int_{A_r}R^2\dd\right)^{1/2}.
 \end{aligned}
\end{equation}
The established global integrability gives
\[
 0\leq\int_{A_r}R^2\dd
 \leq\int_{M\setminus B(o,r)}R^2\dd\longrightarrow0.
\]
Letting $r\to\infty$ in \eqref{eq:surface-annular-estimate}
therefore yields $\int_MR^2\dd=0$. Smoothness implies
$R\equiv0$, and the nonnegative Ricci tensor, having zero trace,
vanishes identically.
\end{proof}

\begin{rem}\label{rem:riemannian-four}
Corollary \ref{cor:kahler-surface} is also valid for $4$-dimensional manifolds with pinching condition if $\varepsilon>1/6$. Indeed,  let $(M^4,g)$ be a connected complete Riemannian manifold satisfying $\Ric\geq\varepsilon Rg\geq0$ for a fixed
$\varepsilon>1/6$. Then $M$ is compact or Ricci-flat.
\end{rem}
Assume that $M$ is noncompact with pinching constant $\varepsilon$ satisfying $1/6<\varepsilon\leq1/4$.
Put $E=\Ric-Rg/4$ and $q=1-4\varepsilon$.
The Ricci eigenvalue bound and the ordinary Bochner formula give
\[
 |E|^2\leq\frac34q^2R^2,\qquad
 \norm{(\nabla^2u)_0}_2^2\leq\frac34\norm{Lu}_2^2,
\]
where $(\nabla^2u)_0=\nabla^2u-(\Delta u)g/4$.
Repeating the proof of Proposition~\ref{prop:quadratic},
with $X=\norm{\eta R}_2$ and
$G^2=\int_M R|\nabla\eta|^2\dd$, yields
\[
 \left(3\varepsilon-\frac12\right)X^2
 \leq\frac{1/2-\varepsilon}{\sqrt{\varepsilon}}XG.
\]
Thus $\int_M\eta^2R^2\dd\leq C_\varepsilon
\int_M R|\nabla\eta|^2\dd$, and the ball iteration and annular
argument in Corollary~\ref{cor:kahler-surface} imply $R\equiv0$.
The restriction $\varepsilon>1/6$ comes from the positivity of
$3\varepsilon-1/2$; the present argument does not give strict
absorption for $0<\varepsilon\leq1/6$.
On a K\"ahler surface, the paired Ricci eigenvalues and the identity
$\norm{H^+_0u}_2^2=\frac14\norm{Lu}_2^2$ improve the interior
coefficient from $3q/4$ to $q/4$. The absorption gap is then
$(1-q)/4=\varepsilon>0$, which is the reason why the K\"ahler conclusion
holds for every positive pinching constant.

\section{Mixed Ricci forms and higher-power curvature estimates}\label{sec:mixed-ricci}

For complete K\"ahler manifolds of complex dimension $m\geq3$,
the quadratic estimate for the scalar curvature alone does not
suffice to complete our Ricci pinching argument. We therefore
study mixed Ricci forms and derive higher-power integral estimates
for the scalar curvature.

\subsection*{Regularized mixed Ricci forms}

Fix $1\leq k\leq m-2$ and $\delta>0$, and define
\begin{equation}\label{eq:mixed-definitions}
 \Theta_{k,\delta}
 =\frac{(\rho+\delta\omega)^k\wedge\omega^{m-k-2}}{(m-k-2)!},
 \qquad Q_{k,\delta}(U,V)dV_g=U\wedge V\wedge\Theta_{k,\delta}.
\end{equation}
For this fixed pair $(k,\delta)$, abbreviate $Q_{k,\delta}$ to $Q$
and set
\begin{equation}\label{eq:mixed-trace}
 h=Q(\omega,\omega),\qquad
 a=\frac{Q(\rho,\omega)}{h},\qquad
 E=\rho-a\omega,\qquad d\mu=h\,dV_g.
\end{equation}
Here a real $(1,1)$-form $U$ is called \emph{primitive relative
to the mixed pairing} if
\[
 U\wedge\omega\wedge\Theta_{k,\delta}=0,
 \quad\text{equivalently}\quad Q(U,\omega)=0.
\]
This is the mixed trace-free condition. Since $h>0$, the choice of
$a$ removes the $\omega$ component of $\rho$ in this pairing:
$Q(E,\omega)=Q(\rho,\omega)-ah=0$.
Thus $E$ is primitive in this sense; this need not mean
$\Lambda_\omega E=0$. The associated pointwise
Hodge--Riemann relation is recalled in Appendix~\ref{app:mixed-hr}.
These quantities are smooth, even where $R=0$, because $\delta>0$.

\begin{lemma}\label{lem:mixed-algebra}
The preceding definitions satisfy
\begin{equation}\label{eq:mixed-comparisons}
 c(R+\delta)^k\leq h\leq C(R+\delta)^k,\qquad
 \varepsilon R\leq a\leq\frac{R}{2m}.
\end{equation}
On the space $Q(U,\omega)=0$, the form $-Q$ is positive definite,
with the uniform bounds
\begin{equation}\label{eq:primitive-bounds}
 -Q(U,U)\geq c h|U|^2,\qquad
 |Q(U,V)|\leq C h|U||V|,
 \qquad -Q(E,E)\leq q^2a^2h.
\end{equation}
The middle inequality holds for arbitrary real $(1,1)$-forms.
For smooth real functions, define the differential expression
$P_0u=-h^{-1}Q(dd^cu,\omega)$. There is a smooth positive symmetric
tensor $D_{k,\delta}$ such that
\begin{equation}\label{eq:mixed-divergence}
 P_0u:=-h^{-1}Q(dd^cu,\omega)
       =-h^{-1}\Div(D_{k,\delta}\nabla u),\qquad
 c h g\leq D_{k,\delta}\leq C h g.
\end{equation}
For real smooth $f,u$, the boundary density satisfies
\begin{equation}\label{eq:boundary-density}
 |df\wedge d^cu\wedge\rho\wedge\Theta_{k,\delta}|
 \leq C R h|\nabla f||\nabla u|\,dV_g.
\end{equation}
\end{lemma}

\begin{proof}
Write $\ell_i=\lambda_i+\delta$ in the frame
\eqref{eq:eigenvalues}. The mixed pairing and its weight are
\begin{equation}\label{eq:mixed-expansion}
 \begin{aligned}
 Q(U,V)&=k!\sum_{i<j}\sigma_k(\ell\mid i,j)\cdot\bigl(U_{i\bar i}V_{j\bar j}
        +U_{j\bar j}V_{i\bar i}
        -2\operatorname{Re}(U_{i\bar j}\overline{V_{i\bar j}})\bigr),
 \end{aligned}
\end{equation}
\begin{equation}\label{eq:h-formula}
 h=k!(m-k)(m-k-1)\sigma_k(\ell).
\end{equation}
Set
\begin{equation}\label{eq:D-formula}
 D_i=k!(m-k-1)\sigma_k(\ell\mid i).
\end{equation}
Here $\sigma_k$ denotes the elementary symmetric polynomial of
degree $k$, with $\sigma_0=1$, and the indices following the vertical bar indicate
the variables to be omitted. Explicitly, for distinct indices
$i,j\in\{1,\ldots,m\}$,
\[
\begin{aligned}
\sigma_k(\ell\mid i)
&=\sum_{\substack{I\subset\{1,\ldots,m\}\setminus\{i\}\\ |I|=k}}
  \prod_{a\in I}\ell_a,\\
\sigma_k(\ell\mid i,j)
&=\sum_{\substack{I\subset\{1,\ldots,m\}\setminus\{i,j\}\\ |I|=k}}
  \prod_{a\in I}\ell_a.
\end{aligned}
\]
Thus these are the elementary symmetric polynomials in the
remaining variables after deleting $\ell_i$, or both $\ell_i$
and $\ell_j$, respectively. We use the convention
$\sigma_0(\ell\mid i)=\sigma_0(\ell\mid i,j)=1$.Then $Q(U,\omega)=\sum_iD_iU_{i\bar i}$ and $\sum_iD_i=h$.
The detailed calculations for
\eqref{eq:mixed-expansion}--\eqref{eq:D-formula} are given in
\eqref{eq:app-mixed-expansion}--\eqref{eq:app-D-formula} of
Appendix~\ref{app:mixed-hr}.

We construct $D_{k,\delta}$ intrinsically by
\[
 D_{k,\delta}(X,Y)\,dV_g
   =-X^\flat\wedge(Y^\flat\circ J)
                         \wedge\omega\wedge\Theta_{k,\delta},
\]
where $X^\flat=g(X,\cdot)$. The expansion in the Ricci eigenframe is
\[
 D_{k,\delta}
   =\sum_iD_i(\theta_i\otimes\theta_i+\psi_i\otimes\psi_i).
\]
Thus it is symmetric and positive definite, while the intrinsic
formula shows smoothness even at repeated eigenvalues and
independence of the chosen frame.
Since $\varepsilon(R+\delta)\leq\ell_i\leq R+\delta$, the
positive polynomial formulas give
$c(R+\delta)^k\leq D_i\leq C(R+\delta)^k$ and the same bounds for $h$.
This proves the bound for $h$ in \eqref{eq:mixed-comparisons}
and $chg\leq D_{k,\delta}\leq Chg$ in \eqref{eq:mixed-divergence}.

It remains to identify $P_0$ in divergence form. Since
$d\rho=d\omega=0$, the form $\omega\wedge\Theta_{k,\delta}$
is closed. For real $f\in C_c^\infty(M)$ and smooth real $u$,
Stokes' theorem and $d^cu=-du\circ J$ give
\[
 \begin{aligned}
 \int_MfhP_0u\dd
 &=-\int_M f\,dd^cu\wedge\omega\wedge\Theta_{k,\delta}\\
 &=\int_Mdf\wedge d^cu\wedge\omega\wedge\Theta_{k,\delta}
   =\int_MD_{k,\delta}(\nabla f,\nabla u)\dd.
 \end{aligned}
\]
Integration by parts in the last expression proves
$hP_0u=-\Div(D_{k,\delta}\nabla u)$ distributionally, hence
pointwise for smooth $u$. This proves \eqref{eq:mixed-divergence}.

Next, the formula for the mixed trace gives
\[
 a=\frac{\sum_i\lambda_iD_i}{\sum_iD_i}\geq\varepsilon R.
\]
For the upper bound, expansion of \eqref{eq:D-formula} gives
\[
 D_i-D_j=k!(m-k-1)(\lambda_j-\lambda_i)
                         \sigma_{k-1}(\ell\mid i,j).
\]
Consequently,
\[
 m\sum_i\lambda_iD_i-\left(\sum_i\lambda_i\right)\left(\sum_iD_i\right)
 =\sum_{i<j}(\lambda_i-\lambda_j)(D_i-D_j)\leq0,
\]
which proves $a\leq R/(2m)$.

For the primitive estimate, apply
Corollary~\ref{cor:uniform-mixed-hr} of Appendix~\ref{app:mixed-hr}
pointwise to $A=(\rho+\delta\omega)/(R+\delta)$.
Its eigenvalues lie in $[\varepsilon,1]$, and both $Q$ and $h$
are $(R+\delta)^k$ times the corresponding normalized quantities.
This proves the uniform lower bound in \eqref{eq:primitive-bounds}.
Formula~\eqref{eq:mixed-expansion} gives the upper bound.

For the sharper bound on $E$, let
$\Gamma=\rho-\varepsilon R\omega\geq0$ and
$b=a-\varepsilon R\geq0$. Then $E=\Gamma-b\omega$ and
$Q(\Gamma,\omega)=bh$. Positivity of wedge products of semipositive
forms gives $Q(\Gamma,\Gamma)\geq0$, so
\[
 -Q(E,E)=b^2h-Q(\Gamma,\Gamma)\leq b^2h.
\]
Since $a\leq R/(2m)$, one has $b\leq(1-2m\varepsilon)a=qa$.
This proves the last bound in \eqref{eq:primitive-bounds}.
Only pointwise positivity was used here; no closedness of
$\Gamma$ is required.

Finally, expand $\rho=\sum_j\lambda_j\alpha_j$ and use the
complementary-index expansion of $\Theta_{k,\delta}$ from
Appendix~\ref{app:mixed-hr}. Only the $\alpha_i$ component of
$df\wedge d^cu$ can be completed by the other factors. Thus
\[
 \begin{aligned}
 df\wedge d^cu\wedge\rho\wedge\Theta_{k,\delta}
 &=\sum_i b_i\bigl((e_if)(e_iu)+(Je_if)(Je_iu)\bigr)\,dV_g,\\
 b_i&=k!\sum_{j\ne i}\lambda_j\sigma_k(\ell\mid i,j).
 \end{aligned}
\]
Since $0\leq\lambda_j\leq R/2$ and
$k!\sum_{j\ne i}\sigma_k(\ell\mid i,j)=D_i$, we have
$0\leq b_i\leq(R/2)D_i$. Cauchy--Schwarz in this frame yields
\[
 \begin{aligned}
 |df\wedge d^cu\wedge\rho\wedge\Theta_{k,\delta}|
 &\leq\frac R2
   \bigl[D_{k,\delta}(\nabla f,\nabla f)
         D_{k,\delta}(\nabla u,\nabla u)\bigr]^{1/2}\,dV_g\\
 &\leq CRh|\nabla f||\nabla u|\,dV_g.
 \end{aligned}
\]
The last step uses $D_{k,\delta}\leq Chg$ and proves
\eqref{eq:boundary-density}.
\end{proof}

We now pass to the operator on $L^2(M,d\mu)$, where
$d\mu=h\,dV_g$. Extend $P_0$ in \eqref{eq:mixed-divergence}
complex linearly and restrict its domain to $C_c^\infty(M)$.
Since $h$ is smooth and positive and
$chg\leq D_{k,\delta}\leq Chg$, Lemma~\ref{lem:operator}
shows that this operator is essentially self-adjoint.
Henceforth,
\[
 P:=\overline{P_0}
 \quad\text{on }L^2(M,d\mu)
\]
denotes its nonnegative self-adjoint realization. Its domain is
the distributional domain described in Lemma~\ref{lem:operator},
and $Pu=P_0u$ on $C_c^\infty(M)$. Write $\norm{\cdot}_\mu$
for the $L^2(d\mu)$ norm and set
\[
 \mathcal E_{k,\delta}(u)
   =\int_M D_{k,\delta}(\nabla u,\nabla u)\dd.
\]
In particular, this energy is comparable to
$\int_M(R+\delta)^k|\nabla u|^2\dd$ with constants independent
of $\delta$.

We next establish the weighted energy and resolvent pairing
identities that will be used to estimate higher powers of the
scalar curvature.

\begin{proposition}\label{prop:mixed-energy}
Fix $k,\delta$ as above, and let $P$ be the self-adjoint operator
just defined. The following assertions hold.
\begin{enumerate}[label=\textup{(\roman*)},leftmargin=*]
\item Let $u\in\Dom(P)$ be real and set $v=Pu$.
Then the distributional form $dd^cu$ has a representative in
$L^2(h\,dV_g)$. Defining $\beta=dd^cu+v\omega$ with this
representative, one has $Q(\beta,\omega)=0$ almost everywhere and
\begin{equation}\label{eq:mixed-energy}
 \int_M-Q(\beta,\beta)\dd=\norm v_\mu^2,\qquad
 \int_Mh|dd^cu|^2\dd\leq C\norm v_\mu^2.
\end{equation}
\item Let $\eta\in C_c^\infty(M)$ be real and let $\tau>0$.
Set $F=\eta a$, $u=(P+\tau)^{-1}F$, and $v=Pu$.
Then $u$ is real and belongs to $\Dom(P)$, so \textup{(i)} applies.
For the resulting $\beta$, the pairing identity is
\begin{equation}\label{eq:mixed-pairing}
 \begin{aligned}
 \ip Fv_\mu
   &=\int_M\eta Q(E,\beta)\dd+\mathcal B_{k,\delta}(\eta,u),\\
 \mathcal B_{k,\delta}(\eta,u)
   &:=\int_Md\eta\wedge d^cu\wedge\rho\wedge\Theta_{k,\delta}.
 \end{aligned}
\end{equation}
Moreover,
\begin{align}
 \left|\int_M\eta Q(E,\beta)\dd\right|
       &\leq q\norm F_\mu\norm v_\mu,\label{eq:strict-absorption}\\
 |\mathcal B_{k,\delta}(\eta,u)|
       &\leq C\left(\int_MR^2h|\nabla\eta|^2\dd\right)^{1/2}
                         \mathcal E_{k,\delta}(u)^{1/2}.
                         \label{eq:boundary-energy}
\end{align}
\end{enumerate}
\end{proposition}

\begin{proof}
\emph{Part \textup{(i)}.}
First let $u\in C_c^\infty(M)$ be real, and write
$A=dd^cu$, $v=Pu$, and $\beta=A+v\omega$.
Since $dA=d\Theta_{k,\delta}=0$, the product rule gives
\[
 d\bigl(d^cu\wedge A\wedge\Theta_{k,\delta}\bigr)
       =A\wedge A\wedge\Theta_{k,\delta}.
\]
The form inside the derivative is compactly supported, so Stokes'
theorem yields
\[
 \int_M Q(A,A)\dd
       =\int_M d\bigl(d^cu\wedge A\wedge\Theta_{k,\delta}\bigr)=0.
\]
By the definition of $P$ on smooth functions,
$Q(A,\omega)=-hv$, while $Q(\omega,\omega)=h$. Hence
\[
 \begin{aligned}
 Q(\beta,\omega)&=Q(A,\omega)+vh=0,\\
 Q(\beta,\beta)&=Q(A,A)+2vQ(A,\omega)+v^2h
                         =Q(A,A)-hv^2.
 \end{aligned}
\]
Integrating the second equality proves
$\int_M-Q(\beta,\beta)\dd=\norm v_\mu^2$.
Because $\beta$ is primitive, the coercivity in
\eqref{eq:primitive-bounds} then gives
\[
 \int_Mh|\beta|^2\dd\leq c^{-1}\norm v_\mu^2.
\]
Using $A=\beta-v\omega$ and $|\omega|^2=m$, we obtain
\[
 \int_Mh|dd^cu|^2\dd
 \leq2\int_Mh|\beta|^2\dd+2m\norm v_\mu^2
 \leq C\norm v_\mu^2.
\]
Now let $u\in\Dom(P)$ be real and set $v=Pu$.
Since $C_c^\infty(M)$ is a graph core, choose real
$u_j\in C_c^\infty(M)$ such that $u_j\to u$ and
$v_j:=Pu_j\to v$ in $L^2(d\mu)$. Real approximants can be
obtained by taking real parts, since $P$ has real coefficients.
Applying the preceding estimate to $u_j-u_\ell$ gives
\[
 \int_Mh|dd^cu_j-dd^cu_\ell|^2\dd
       \leq C\norm{v_j-v_\ell}_\mu^2\longrightarrow0.
\]
Consequently $dd^cu_j$ converges to a real $(1,1)$-form $A$
in $L^2(h\,dV_g)$. On each compact set $h$ has a positive
lower bound, so $u_j\to u$ and $dd^cu_j\to A$ locally in
$L^2$. Distributional differentiation is continuous, and thus
$dd^cu_j\to dd^cu$ in $\mathcal D'$. These two limits identify
$A$ as the weighted $L^2$ representative of $dd^cu$.

Set $\beta=A+v\omega$ and $\beta_j=dd^cu_j+v_j\omega$.
The preceding convergences and $|\omega|^2=m$ imply
$\beta_j\to\beta$ in $L^2(h\,dV_g)$.
We next pass the primitive condition to this limit.
Bilinearity and $Q(\beta_j,\omega)=0$ give
\[
 Q(\beta,\omega)
   =Q(\beta-\beta_j,\omega)+Q(\beta_j,\omega)
   =Q(\beta-\beta_j,\omega).
\]
The general bound for $Q$ in \eqref{eq:primitive-bounds} yields
\[
 \int_M\frac{|Q(\beta,\omega)|^2}{h}\dd
       \leq C^2m\int_Mh|\beta-\beta_j|^2\dd
       \longrightarrow0.
\]
The left side is independent of $j$ and nonnegative; it is
therefore zero, proving $Q(\beta,\omega)=0$ almost everywhere.
The same bound and Cauchy--Schwarz give
\[
 \begin{aligned}
 \left|\int_M\bigl[Q(\beta_j,\beta_j)-Q(\beta,\beta)\bigr]\dd\right|
 &\leq C\norm{\beta_j-\beta}_{L^2(d\mu)}\\
 &\quad\cdot\bigl(\norm{\beta_j}_{L^2(d\mu)}
                         +\norm\beta_{L^2(d\mu)}\bigr)
       \longrightarrow0.
 \end{aligned}
\]
Here the quadratic integrals are finite because
$|Q(U,U)|\leq Ch|U|^2$. Passing to the limit in
$\int_M-Q(\beta_j,\beta_j)\dd=\norm{v_j}_\mu^2$
proves the first identity in \eqref{eq:mixed-energy}; strong
convergence of $dd^cu_j$ gives the second inequality.

\smallskip\goodbreak
\noindent\emph{Part \textup{(ii)}.}
The resolvent gives $u\in\Dom(P)$. Real coefficients and
uniqueness make $u$ real. Since $F=\eta a$ is smooth, the
distributional equation $(P_0+\tau)u=F$ and local elliptic
regularity give $u\in C^\infty(M)$. Part \textup{(i)}
supplies the global weighted bounds.

To apply Stokes' theorem, observe that
$d(\rho\wedge\Theta_{k,\delta})=0$ because
$d\rho=d\omega=0$. The $(2m-1)$-form
$\eta\,d^cu\wedge\rho\wedge\Theta_{k,\delta}$ is compactly
supported, even though $u$ need not be. The product rule and
Stokes' theorem therefore give
\[
 \begin{aligned}
 0&=\int_Md\bigl(\eta\,d^cu\wedge\rho\wedge\Theta_{k,\delta}\bigr)\\
  &=\int_Md\eta\wedge d^cu\wedge\rho\wedge\Theta_{k,\delta}
       +\int_M\eta\,dd^cu\wedge\rho\wedge\Theta_{k,\delta}\\
  &=\mathcal B_{k,\delta}(\eta,u)
       +\int_M\eta Q(\rho,dd^cu)\dd.
 \end{aligned}
\]
Thus $\int_M\eta Q(\rho,dd^cu)\dd=-\mathcal B_{k,\delta}(\eta,u)$.
Since $E=\rho-a\omega$, $Q(E,\omega)=0$, and
$Q(\omega,dd^cu)=-hv$, we obtain
\[
 \begin{aligned}
 \int_M\eta Q(E,\beta)\dd
 &=\int_M\eta Q(E,dd^cu)\dd\\
 &=\int_M\eta Q(\rho,dd^cu)\dd
       -\int_M\eta aQ(\omega,dd^cu)\dd\\
 &=-\mathcal B_{k,\delta}(\eta,u)+\int_M\eta ahv\dd\\
 &=-\mathcal B_{k,\delta}(\eta,u)+\ip Fv_\mu.
 \end{aligned}
\]
This proves \eqref{eq:mixed-pairing}. Here Stokes' theorem uses
only the closed factors $\rho$ and $\Theta_{k,\delta}$; the
coefficient $a$ enters through the subsequent algebraic expansion.

Both $E$ and $\beta$ are primitive, so pointwise
Cauchy--Schwarz for the positive definite form $-Q$, followed
by the integral Cauchy--Schwarz inequality, gives
\[
 \begin{aligned}
 \left|\int_M\eta Q(E,\beta)\dd\right|
 &\leq\left(\int_M\eta^2[-Q(E,E)]\dd\right)^{1/2}
      \left(\int_M-Q(\beta,\beta)\dd\right)^{1/2}\\
 &\leq q\left(\int_M\eta^2a^2h\dd\right)^{1/2}\norm v_\mu
       =q\norm F_\mu\norm v_\mu.
 \end{aligned}
\]
Here we used $-Q(E,E)\leq q^2a^2h$ and part \textup{(i)}.
Finally, \eqref{eq:boundary-density}, Cauchy--Schwarz, and
$D_{k,\delta}\geq chg$ give
\[
 \begin{aligned}
 |\mathcal B_{k,\delta}(\eta,u)|
 &\leq C\int_MRh|\nabla\eta||\nabla u|\dd\\
 &\leq C\left(\int_MR^2h|\nabla\eta|^2\dd\right)^{1/2}
                         \mathcal E_{k,\delta}(u)^{1/2}.
 \end{aligned}
\]
The energy is finite by Lemma~\ref{lem:operator}, and
\eqref{eq:boundary-energy} follows.
\end{proof}

\subsection*{Localized estimates for powers of the scalar curvature}

We can now derive the following higher-power integral estimates
for the scalar curvature.

\begin{proposition}\label{prop:hierarchy}
Let $(M^{2m},g,J)$ be the K\"ahler manifold in Theorem \ref{thm:main}.
Then, for every integer $2\leq p\leq m$, every $\delta>0$, and
every real $\eta\in C_c^\infty(M)$,
\begin{equation}\label{eq:hierarchy}
 \int_M\eta^2R^p\dd
 \leq C_{m,\varepsilon,p}\left(
      \delta^{-1}\int_M(R+\delta)^p|\nabla\eta|^2\dd
                    +\delta^p\int_M\eta^2\dd\right).
\end{equation}
\end{proposition}

\begin{proof}
\noindent\emph{The quadratic base case and the resolvent setup.}
The quadratic estimate \eqref{eq:quadratic} and
$R\leq(R+\delta)^2/(4\delta)$ prove the case $p=2$.
We prove the estimates for higher powers
on the same manifold by induction, using $p=k+2$ and the mixed
operator with index $k$.

Fix $1\leq k\leq m-2$, $\delta>0$, and $\eta$.
Use the notation \eqref{eq:mixed-trace}, and put
\begin{equation}\label{eq:resolvent-choice}
 F=\eta a,\qquad u=(P+\tau)^{-1}F,\qquad v=Pu,
 \qquad \tau=\xi\delta,
\end{equation}
where the positive number $\xi$ will be chosen below, independently
of $\delta$ and $\eta$. If $F=0$, then
\eqref{eq:mixed-comparisons} implies $\eta R=0$, and the desired
estimate is immediate. We may therefore suppose $\norm F_\mu>0$.
By \eqref{eq:spectral},
\begin{equation}\label{eq:fixed-mass-bounds}
 \norm v_\mu\leq\norm F_\mu,\qquad
 \tau\norm u_\mu\leq\norm F_\mu,\qquad
 \mathcal E_{k,\delta}(u)\leq\frac{\norm F_\mu^2}{4\tau}.
\end{equation}
The identity $F=v+\tau u$ gives
$\norm F_\mu^2=\ip Fv_\mu+\tau\ip Fu_\mu$.
Combining this with
\eqref{eq:mixed-pairing}--\eqref{eq:strict-absorption} and
$\norm v_\mu\leq\norm F_\mu$ yields
\begin{equation}\label{eq:defect-start}
 (1-q)\norm F_\mu^2
 \leq |\mathcal B_{k,\delta}(\eta,u)|
                       +\tau\ip Fu_\mu.
\end{equation}

We first estimate the term $\tau\langle F,u\rangle_\mu$
by splitting the integral into the high-curvature region
$\{R\geq\delta\}$ and the low-curvature region $\{R<\delta\}$.

\par\smallskip
\noindent\emph{The high-curvature region $\{R\geq\delta\}$ for $k\geq2$.}
We first estimate the resolvent remainder over $\{R\geq\delta\}$.
Suppose $k\geq2$ and assume that \eqref{eq:hierarchy} has been
proved at order $k$, for every positive scale. Fix $0<t<1$.
Before applying it to the noncompactly supported $u$, we justify
that extension separately. Choose real $u_j\in C_c^\infty(M)$
converging to $u$ in the graph norm of $P$. Its energy convergence
follows from Lemma~\ref{lem:operator}. For every difference
$z=u_j-u_\ell$, the already established order-$k$ estimate at
scale $t\delta$ gives
\begin{align*}
 \int_M R^kz^2\dd
 &\leq C_k\left((t\delta)^{-1}
          \int_M(R+t\delta)^k|\nabla z|^2\dd
                   +(t\delta)^k\int_M z^2\dd\right)\\
 &\leq C_k\left((t\delta)^{-1}\mathcal E_{k,\delta}(z)
                         +t^k\|z\|_\mu^2\right)
 \longrightarrow0.
\end{align*}
The second line uses $D_{k,\delta}\geq c(R+\delta)^kg$,
$t<1$, and $h\geq c\delta^k$. Hence $R^{k/2}u_j$ converges
to some $G\in L^2(M,dV_g)$ by completeness.

\goodbreak
To identify $G$, fix a compact set $K\subset M$. Since $R$ is
bounded on $K$ and $h$ has a positive lower bound there,
\[
 \|R^{k/2}(u_j-u)\|_{L^2(K)}
 \leq\underbrace{\frac{(\sup_K R)^{k/2}}{(\inf_K h)^{1/2}}}
                  _{C_{\mathrm{loc}}(K,\delta)}
                          \|u_j-u\|_\mu\longrightarrow0.
\]
This local factor is finite for fixed $K,\delta$; it need not be
uniform as $\delta$ varies. The global convergence to $G$ also
holds after restriction to $K$.
Uniqueness of the $L^2(K)$ limit therefore gives
$G=R^{k/2}u$ on $K$. A compact exhaustion of $M$ yields this
identity almost everywhere on $M$, and consequently
\[
 \int_MR^ku^2\dd=\|G\|_2^2
       =\lim_{j\to\infty}\int_MR^ku_j^2\dd<\infty.
\]
Moreover, graph-core convergence and Lemma~\ref{lem:operator}
give $\mathcal E_{k,\delta}(u_j-u)\to0$, hence
$\mathcal E_{k,\delta}(u_j)\to\mathcal E_{k,\delta}(u)$.
Applying the order-$k$ estimate to each $u_j$ and using these
convergences together with $\|u_j\|_\mu\to\|u\|_\mu$ gives
\begin{equation}\label{eq:lower-order-u}
 \int_M R^ku^2\dd
 \leq C_k\left((t\delta)^{-1}\mathcal E_{k,\delta}(u)
                           +t^k\norm u_\mu^2\right).
\end{equation}
Here and below $C_k$ is determined only by $m,\varepsilon,k$
and the already established order-$k$ constant. The supports of
$u_j$ may vary with $j$. No common compact support is required
for this graph-core argument; the common-support approximation
used earlier serves the different purpose of admitting compactly
supported Lipschitz cutoffs.

Since $h\leq CR^k$ on $\{R\geq\delta\}$, Cauchy--Schwarz gives
\[
 \tau\left|\ip{F\mathbf1_{\{R\geq\delta\}}}{u}_\mu\right|
 \leq C\tau\norm F_\mu
                    \left(\int_M R^ku^2\dd\right)^{1/2}.
\]
This Cauchy--Schwarz estimate is valid also for $k=1$.
The next step, however, uses the already established order-$k$
estimate through \eqref{eq:lower-order-u}.
The induction hypothesis does not apply when $k=1$, since the
family \eqref{eq:hierarchy} starts at order two and provides no
order-one estimate to invoke. We therefore complete the argument
for $k\geq2$ first, and then treat $k=1$ separately below using
the quadratic estimate \eqref{eq:quadratic} from
Proposition~\ref{prop:quadratic}.

For $k\geq2$, \eqref{eq:lower-order-u} and
\eqref{eq:fixed-mass-bounds} now give the following bounds with
one fixed constant $M_k\geq1$.
This constant depends only on $m,\varepsilon,k$ and the already
established order-$k$ constant, independently of
$t,\xi,\delta,\eta$: its value is fixed before choosing $t$ and $\xi$.
\begin{align}
 \tau\left|\ip{F\mathbf1_{\{R\geq\delta\}}}{u}_\mu\right|
 &\leq M_k\left(\sqrt{\frac\xi t}+t^{k/2}\right)
                                      \norm F_\mu^2\\
 &\leq\frac{1-q}{4}\norm F_\mu^2.
                                      \label{eq:high-defect}
\end{align}
An explicit choice giving the last inequality is
\[
 t=\left(\frac{1-q}{8M_k}\right)^{2/k},\qquad
 \xi=t\left(\frac{1-q}{8M_k}\right)^2.
\]
Indeed, $0<t<1$, $\xi>0$, and
$M_kt^{k/2}=M_k\sqrt{\xi/t}=(1-q)/8$.
Thus the two terms together use at most $(1-q)/4$ of the gain.
Both choices depend on an established lower-order constant and
are independent of the constant being proved at order $k+2$.

\smallskip
\noindent\emph{The high-curvature region $\{R\geq\delta\}$ for $k=1$.}
Here $p=k+2=3$, and we keep $P=P_{1,\delta}$ and $\delta>0$
fixed while extending the quadratic estimate to $u$.
Proposition~\ref{prop:quadratic} states that every real $z\in C_c^\infty(M)$ satisfies
\[
 \int_M R^2z^2\dd
 \leq C_{\mathrm{quad}}\int_M R|\nabla z|^2\dd,
 \qquad
 C_{\mathrm{quad}}
 =\frac{(1/2-\varepsilon)^2}{2(m-1)^2\varepsilon^3}.
\]
By Lemma~\ref{lem:mixed-algebra}
\[
 c(R+\delta)\leq h\leq C(R+\delta),\qquad
 D_{1,\delta}\geq chg.
\]
In particular, $R\leq Ch$, and hence
\[
 \int_M R|\nabla z|^2\dd
 \leq C\int_M D_{1,\delta}(\nabla z,\nabla z)\dd
 =C\mathcal E_{1,\delta}(z).
\]
All constants in these comparisons are independent of $\delta$.

To apply the quadratic estimate to the noncompactly supported
resolvent solution $u$, use the graph-core assertion of
Lemma~\ref{lem:operator} to choose real $u_j\in C_c^\infty(M)$
such that
\[
 \|u_j-u\|_\mu+\|Pu_j-Pu\|_\mu\longrightarrow0.
\]
The energy identity \eqref{eq:energy} gives
\[
 \begin{aligned}
 \mathcal E_{1,\delta}(u_j-u)
 &=\ip{u_j-u}{P(u_j-u)}_\mu\\
 &\leq\|u_j-u\|_\mu\|Pu_j-Pu\|_\mu
 \longrightarrow0.
 \end{aligned}
\]
The same argument shows
$\mathcal E_{1,\delta}(u_j-u_\ell)\to0$ as $j,\ell\to\infty$.
Applying \eqref{eq:quadratic} to $z=u_j-u_\ell$ therefore yields
\[
 \begin{aligned}
 \|R(u_j-u_\ell)\|_2^2
 &\leq C_{\mathrm{quad}}
       \int_M R|\nabla(u_j-u_\ell)|^2\dd\\
 &\leq C\mathcal E_{1,\delta}(u_j-u_\ell)
 \longrightarrow0.
 \end{aligned}
\]
Thus $Ru_j$ converges to some $W\in L^2(M,dV_g)$.
For each compact set $K$, positivity of $h$ and boundedness of $R$
on $K$ give
\[
 \|R(u_j-u)\|_{L^2(K)}^2
 \leq\frac{\sup_K R^2}{\inf_K h}\|u_j-u\|_\mu^2
 \longrightarrow0.
\]
The local factor is finite for fixed $K,\delta$; no uniform bound
on it as $\delta$ varies is needed. Uniqueness of the local
$L^2$ limit identifies $W=Ru$ on every compact set, and therefore
on $M$. Moreover, the energy comparison above gives
\[
 \|R^{1/2}(\nabla u_j-\nabla u)\|_2^2
 \leq C\mathcal E_{1,\delta}(u_j-u)\longrightarrow0.
\]
We may consequently pass to the limit in \eqref{eq:quadratic}
with test function $u_j$ and obtain
\begin{equation}\label{eq:cubic-base-u}
 \int_M R^2u^2\dd
 \leq C_{\mathrm{quad}}\int_M R|\nabla u|^2\dd
 \leq C\mathcal E_{1,\delta}(u).
\end{equation}
This argument establishes the required integrability of $Ru$;
it does not assume global integrability of the scalar curvature.

Now set $A_\delta=\{R\geq\delta\}$. On this set,
\eqref{eq:mixed-comparisons} with $k=1$ implies
\[
 h\leq C(R+\delta)\leq CR,\qquad
 \frac{h}{R^2}\leq\frac{C}{R}\leq\frac{C}{\delta}.
\]
It follows that
\[
 \int_{A_\delta}u^2h\dd
 \leq\frac{C}{\delta}\int_M R^2u^2\dd.
\]
Cauchy--Schwarz, \eqref{eq:cubic-base-u}, and the energy bound
in \eqref{eq:fixed-mass-bounds} now give
\begin{align*}
 \tau\left|\ip{F\mathbf1_{A_\delta}}{u}_\mu\right|
 &\leq\tau\|F\|_\mu
             \left(\int_{A_\delta}u^2h\dd\right)^{1/2}\\
 &\leq C\tau\delta^{-1/2}\|F\|_\mu
             \left(\int_M R^2u^2\dd\right)^{1/2}\\
 &\leq C\tau\delta^{-1/2}\|F\|_\mu
                      \mathcal E_{1,\delta}(u)^{1/2}\\
 &\leq\frac{C}{2}\sqrt{\frac{\tau}{\delta}}\|F\|_\mu^2
 =B_1\sqrt\xi\,\|F\|_\mu^2.
\end{align*}
Here we used $\tau=\xi\delta$ from
\eqref{eq:resolvent-choice}; the fixed constant
$B_1=B_1(m,\varepsilon)>0$ depends only on the constants in
Proposition~\ref{prop:quadratic} and Lemma~\ref{lem:mixed-algebra}.
In particular, it is independent of $\xi,\delta$, and $\eta$.
Choosing
\[
 \xi=\left(\frac{1-q}{4\max\{1,B_1\}}\right)^2
\]
therefore gives
\[
 \tau\left|\ip{F\mathbf1_{\{R\geq\delta\}}}{u}_\mu\right|
 \leq\frac{1-q}{4}\|F\|_\mu^2.
\]

For $k\geq2$, we first fix $m,\varepsilon,k$ and the constant
in the already established order-$k$ estimate; these data determine
$M_k$. We then choose $t\in(0,1)$ so that
$M_kt^{k/2}\leq(1-q)/8$. With this $t$ fixed, we choose
$\xi>0$ so that $M_k\sqrt{\xi/t}\leq(1-q)/8$.
For $k=1$, we use the choice of $\xi$ given above, which depends
only on $m,\varepsilon$ through the quadratic estimate and the
mixed coefficient bounds, and no parameter $t$ is needed.
In either case, these choices are independent of $\delta$ and
$\eta$ and use only constants already established, so they do
not depend on the order-$(k+2)$ constant being proved.
We now keep these parameters fixed: for every $\delta>0$ and
every real $\eta\in C_c^\infty(M)$, we set $\tau=\xi\delta$
and obtain
\begin{equation}\label{eq:high-absorbed}
 \tau\left|\ip{F\mathbf1_{\{R\geq\delta\}}}{u}_\mu\right|
                  \leq\frac{1-q}{4}\norm F_\mu^2.
\end{equation}

\par\goodbreak\smallskip
\noindent\emph{The low-curvature region $\{R<\delta\}$.}
We now estimate the low-curvature remainder for every
$1\leq k\leq m-2$, including the cubic base case $k=1$.
By \eqref{eq:mixed-comparisons}, on $\{R<\delta\}$ we have
$a\leq CR\leq C\delta$ and
$h\leq C(R+\delta)^k\leq C\delta^k$.
Since $F=\eta a$, we obtain
\begin{equation}\label{eq:low-source}
 \begin{aligned}
 \norm{F\mathbf1_{\{R<\delta\}}}_\mu^2
   &=\int_{\{R<\delta\}}\eta^2a^2h\dd\\
   &\leq C\delta^k\int_{\{R<\delta\}}\eta^2R^2\dd
    \leq C\delta^{k+2}\int_M\eta^2\dd.
 \end{aligned}
\end{equation}
\goodbreak
Use Cauchy--Schwarz and $\tau\norm u_\mu\leq\norm F_\mu$.
Young's inequality in the form
$xy\leq(1-q)x^2/4+y^2/(1-q)$, with
$x=\norm F_\mu$ and
$y=\norm{F\mathbf1_{\{R<\delta\}}}_\mu$, then gives
\begin{equation}\label{eq:low-defect}
 \begin{aligned}
 \tau\left|\ip{F\mathbf1_{\{R<\delta\}}}{u}_\mu\right|
 &\leq\tau\norm{F\mathbf1_{\{R<\delta\}}}_\mu\norm u_\mu\\
 &\leq\norm{F\mathbf1_{\{R<\delta\}}}_\mu\norm F_\mu\\
 &\leq\frac{1-q}{4}\norm F_\mu^2
       +\frac{1}{1-q}\norm{F\mathbf1_{\{R<\delta\}}}_\mu^2\\
 &\leq\frac{1-q}{4}\norm F_\mu^2
                       +C\delta^{k+2}\int_M\eta^2\dd.
 \end{aligned}
\end{equation}
The two characteristic functions are used only to split integrals;
no derivative is taken of either one.

\par\smallskip
\noindent\emph{The boundary term.}
Finally, \eqref{eq:boundary-energy} and
\eqref{eq:fixed-mass-bounds} first give
\[
 |\mathcal B_{k,\delta}(\eta,u)|
 \leq\frac{C}{\sqrt{\xi\delta}}\|F\|_\mu
             \left(\int_M R^2h|\nabla\eta|^2\dd\right)^{1/2}.
\]
Young's inequality, with coefficient $(1-q)/4$ for $\|F\|_\mu^2$,
therefore implies
\begin{equation}\label{eq:boundary-absorbed}
 |\mathcal B_{k,\delta}(\eta,u)|
 \leq\frac{1-q}{4}\norm F_\mu^2
                 +\frac C\delta\int_M R^2h|\nabla\eta|^2\dd.
\end{equation}
The constant in \eqref{eq:boundary-absorbed} is allowed to depend
on $\xi$ and $1-q$. Since $\xi$ has already been fixed by
the stated data, this is still a constant depending only on
$m,\varepsilon,k$ and the previously established lower-order
constant, never on $\delta$ or $\eta$.
\par\smallskip
\noindent\emph{Absorption and completion of the induction.}
Insert \eqref{eq:high-absorbed}, \eqref{eq:low-defect}, and
\eqref{eq:boundary-absorbed} into \eqref{eq:defect-start}.
Absorbing the three quarters of the strict gain yields
\[
 \norm F_\mu^2
 \leq\frac C\delta\int_M R^2h|\nabla\eta|^2\dd
                      +C\delta^{k+2}\int_M\eta^2\dd.
\]
Since
\[
 \norm F_\mu^2=\int_M\eta^2a^2h\dd
       \geq c\int_M\eta^2R^{k+2}\dd,
 \qquad R^2h\leq C(R+\delta)^{k+2},
\]
this is \eqref{eq:hierarchy} at order $k+2$.
In particular, for $k=1$, \eqref{eq:mixed-comparisons} gives
$a^2h\geq cR^2(R+\delta)\geq cR^3$ and
$R^2h\leq CR^2(R+\delta)\leq C(R+\delta)^3$.
Thus the preceding estimate explicitly yields
\[
 \int_M\eta^2R^3\dd
 \leq C_{m,\varepsilon,3}\left(
 \delta^{-1}\int_M(R+\delta)^3|\nabla\eta|^2\dd
 +\delta^3\int_M\eta^2\dd\right).
\]
This is the cubic base case $p=3$.
Strong induction then establishes every order through $m$:
the even orders start at two,
and the odd orders start at three. The ambient complex dimension
has remained $m$ throughout, so no assertion in a lower dimension
has been used as an induction hypothesis.
At each order all auxiliary parameters have been fixed using
only previously established constants, so the resulting
$C_{m,\varepsilon,p}$ is uniform for all $\delta>0$ and all
compactly supported tests.
\end{proof}

\begin{rem}\label{rem:curvature-split}
The decomposition into $\{R\geq\delta\}$ and $\{R<\delta\}$ in
the proof of Proposition~\ref{prop:hierarchy} compares the scalar
curvature with the regularization scale $\delta$. Its purpose is
to obtain a sufficiently small coefficient for the resolvent
remainder in \eqref{eq:defect-start}. Indeed, Cauchy--Schwarz and
\eqref{eq:fixed-mass-bounds} alone give
\[
 \tau|\ip Fu_\mu|
 \leq\tau\norm F_\mu\norm u_\mu
 \leq\norm F_\mu^2.
\]
The coefficient $1$ cannot be strictly absorbed into $1-q$.
Moreover, the factor $\tau$ cancels in this estimate, so choosing
$\tau=\xi\delta$ with small $\xi$ yields no improvement in this bound.

On $\{R\geq\delta\}$, the comparison $h\leq CR^k$ permits us
to control the weighted norm of $u$ using a previously established
curvature estimate. For $k\geq2$ this is the order-$k$ estimate
\eqref{eq:lower-order-u}; for $k=1$ we instead use the quadratic
estimate of Proposition~\ref{prop:quadratic}, extended to $u$ in
\eqref{eq:cubic-base-u}. Choosing the auxiliary parameters then
gives the small coefficient in \eqref{eq:high-absorbed}.

On $\{R<\delta\}$, we control the source $F=\eta a$ directly.
The bounds $a\leq C\delta$ and $h\leq C\delta^k$ imply
\[
 \norm{F\mathbf1_{\{R<\delta\}}}_\mu^2
 =\int_{\{R<\delta\}}\eta^2a^2h\dd
 \leq C\delta^{k+2}\int_M\eta^2\dd.
\]
The basic resolvent bound and Young's inequality therefore give
\eqref{eq:low-defect} for every $1\leq k\leq m-2$, including
$k=1$, without invoking a lower-order induction hypothesis.
The resulting error is already allowed on the right-hand side
of \eqref{eq:hierarchy}.

Combining \eqref{eq:high-absorbed} and \eqref{eq:low-defect} yields
\[
 \tau|\ip Fu_\mu|
 \leq\frac{1-q}{2}\norm F_\mu^2
       +C\delta^{k+2}\int_M\eta^2\dd.
\]
Thus the remainder uses only half of the absorption gain in
\eqref{eq:defect-start}, leaving room for the boundary estimate
\eqref{eq:boundary-absorbed}. All constants are independent of
$\delta$ and $\eta$. A direct treatment of the full inner product
would need an estimate of this strength.
\end{rem}

\section{Proof of the main theorem}\label{sec:proof}

With Proposition~\ref{prop:hierarchy} established, we now prove
Theorem~\ref{thm:main}. We first obtain $R\in L^m(M)$
and then conclude by an annular cutoff.

\begin{proof}[Proof of Theorem~\ref{thm:main}]
Corollary~\ref{cor:kahler-surface} already establishes
Theorem~\ref{thm:main} in the case $m=2$, it remains to consider
a noncompact $M$ with $m\geq3$. Write
$I_m(r)=\int_{B(o,r)}R^m\dd$. These quantities are finite for
finite $r$, since closed bounded balls are compact and $R$ is smooth.

For $0<s<t$, take $\eta=\eta_{s,t}$ from \eqref{eq:cutoff}.
Proposition~\ref{prop:compact-approximation} extends
\eqref{eq:hierarchy} to this Lipschitz test by strong $W^{1,2}$
approximation with a common compact support: for fixed $\delta$,
all coefficients in the weighted integrals are bounded on that
compact set. We may then use $0\leq\eta\leq1$, $\eta=1$ on
$B(o,s)$, and $|\nabla\eta|\leq(t-s)^{-1}$ almost everywhere.
Applying \eqref{eq:hierarchy} with $p=m$ gives, for every $\delta>0$,
\begin{equation}\label{eq:ball-recursion}
 \begin{aligned}
 I_m(s)&\leq\int_M\eta^2R^m\dd\\
 &\leq\frac{L_m}{\delta(t-s)^2}
                    \int_{B(o,t)}(R+\delta)^m\dd
                          +L_m\delta^mV(t)\\
 &\leq\frac{K_0}{\delta(t-s)^2}I_m(t)
       +K_1\left(\frac{\delta^{m-1}}{(t-s)^2}+\delta^m\right)V(t).
 \end{aligned}
\end{equation}
Here $L_m\geq1$ is a valid constant in \eqref{eq:hierarchy}
at order $p=m$, and the constants in the last line are
\[
 K_0=2^{m-1}L_m,\qquad K_1=(2^{m-1}+1)L_m.
\]
They depend only on $m,\varepsilon$ and are independent of
$s,t,\delta$, and the cutoff. The last line follows from
$(R+\delta)^m\leq2^{m-1}(R^m+\delta^m)$; the term $\delta^m$
in this expansion contributes
$2^{m-1}L_m\delta^{m-1}(t-s)^{-2}V(t)$.
We use these same constants throughout the proof.

To obtain a recursion suitable for iteration, take
$\delta=A(t-s)^{-2}$, where $A>0$ will be specified below.
Substituting this scale into \eqref{eq:ball-recursion} gives
\begin{equation}\label{eq:contracting-recursion}
 \begin{aligned}
 I_m(s)
 &\leq\frac{K_0}{A}I_m(t)
       +K_1(A^{m-1}+A^m)(t-s)^{-2m}V(t)\\
 &=\vartheta I_m(t)+C_A(t-s)^{-2m}V(t),
 \end{aligned}
\end{equation}
where we have written
\[
 \vartheta=\frac{K_0}{A},\qquad
 C_A=K_1(A^{m-1}+A^m).
\]
The iteration below requires $\vartheta2^{2m}<1$.
We therefore choose
\[
 A=2^{2m+1}K_0.
\]
With this choice, $\vartheta=2^{-(2m+1)}$ and
$\vartheta2^{2m}=1/2<1$. The constants $A,\vartheta$, and $C_A$
are now fixed and depend only on $m,\varepsilon$.
Although $C_A$ depends on $A$, its value is not used to choose
$A$; that choice uses only $K_0$.
For a fixed $r>0$, set $s_j=2r-r2^{-j}$, $j\geq0$.
Then $s_0=r$, $s_j\uparrow2r$, and
$s_{j+1}-s_j=r2^{-j-1}$. The corresponding scales are
\[
 \delta_j=A(s_{j+1}-s_j)^{-2}
          =Ar^{-2}2^{2j+2}\longrightarrow\infty.
\]
Every application is valid because \eqref{eq:hierarchy} holds
uniformly for all $\delta>0$, with no smallness requirement.
Iterating \eqref{eq:contracting-recursion} inside this fixed
outer ball gives
\[
 I_m(r)\leq\vartheta^NI_m(s_N)
 +C_A r^{-2m}V(2r)
             \sum_{j=0}^{N-1}\vartheta^j2^{2m(j+1)}.
\]
For this fixed $r$, local smoothness and compactness of the closed
outer ball give $I_m(s_N)\leq I_m(2r)<\infty$ independently of
$N$. Hence $\vartheta^NI_m(s_N)\to0$ as $N\to\infty$.
Only this local finiteness is used in taking the iteration limit.
Moreover,
\[
 \sum_{j=0}^{\infty}\vartheta^j2^{2m(j+1)}
       =\frac{2^{2m}}{1-\vartheta2^{2m}}=2^{2m+1}.
\]
Since $A,C_A,\vartheta$ have been fixed using only $m,\varepsilon$, together with Bishop-Gromov volume comparison \eqref{eq:volume} now gives a bound independent of $r$
\[ I_m(r)\leq C r^{-2m}V(2r)\leq \omega_{2m}C,\]
Monotone convergence as $r\to\infty$ then yields global integrability for $R^m$:
\begin{equation}\label{eq:critical-integrability}
 \int_MR^m\dd<\infty.
\end{equation}
The iteration used only finiteness on the one compact outer ball;
global integrability is its conclusion.

\goodbreak
To force this finite integral to vanish, apply
\eqref{eq:hierarchy} again, now with $\eta=\eta_{r,2r}$ and
$\delta=br^{-2}$, where $b>0$ is arbitrary but fixed.
Set $A_r=\overline{B(o,2r)}\setminus B(o,r)$.
The gradient vanishes almost everywhere outside $A_r$ and
satisfies $|\nabla\eta|\leq r^{-1}$. The choice of $\delta$ gives
\[
 \frac1{\delta r^2}=\frac1b,\qquad
 \frac{\delta^{m-1}}{r^2}=b^{m-1}r^{-2m},\qquad
 \delta^m=b^mr^{-2m}.
\]
Using these identities and the same power expansion, we obtain
\begin{equation}\label{eq:annular}
 \begin{aligned}
 I_m(r)&\leq\frac{L_m}{\delta r^2}
                \int_{A_r}(R+\delta)^m\dd+L_m\delta^mV(2r)\\
 &\leq\frac{K_0}{b}\int_{A_r}R^m\dd
       +K_1\bigl(b^{m-1}+b^m\bigr)r^{-2m}V(2r)\\
 &\leq\frac{K_2}{b}
       \int_{\overline{B(o,2r)}\setminus B(o,r)}R^m\dd
                    +K_3\bigl(b^{m-1}+b^m\bigr),
 \end{aligned}
\end{equation}
where
\[
 K_2=K_0,\qquad K_3=2^{2m}\omega_{2m}K_1.
\]
The second line uses $(R+\delta)^m\leq2^{m-1}(R^m+\delta^m)$;
the last uses Bishop-Gromov volume comparison $V(2r)\leq\omega_{2m}(2r)^{2m}$.
In particular, $K_2,K_3$ are independent of both $b$ and $r$.
Keep $b>0$ fixed at this stage. The global integrability established
in \eqref{eq:critical-integrability} now gives
\[
 0\leq\int_{\overline{B(o,2r)}\setminus B(o,r)}R^m\dd
 \leq\int_{M\setminus B(o,r)}R^m\dd\longrightarrow0
 \qquad(r\to\infty).
\]
Thus $K_2/b$ is fixed when the tail tends to zero.
Monotone convergence on the left of \eqref{eq:annular} gives
\[
 \int_MR^m\dd\leq K_3\bigl(b^{m-1}+b^m\bigr).
\]
This inequality holds for every fixed $b>0$ with the same $K_3$.
Only after completing the limit $r\to\infty$ do we let
$b\downarrow0$.
Since $m\geq2$, both powers on the right vanish. Hence $R=0$
almost everywhere, and smoothness gives $R\equiv0$.
The nonnegative Ricci tensor has zero trace and therefore vanishes
identically. This proves the alternative in the Theorem~\ref{thm:main}.\qedhere
\end{proof}

\appendix
\section{Compactly supported Sobolev approximation}\label{app:approximation}

The following local approximation result justifies the nonsmooth
cutoffs and weak tests used in the proof. Completeness is needed
to make the distance cutoffs compactly supported, but not for the
approximation result itself. All manifolds in this appendix have
no boundary.

\begin{proposition}\label{prop:compact-approximation}
Let $f\in W^{1,2}(M)$ have compact support $K$. There are a compact
set $K'\subset M$ and functions $f_j\in C_c^\infty(M)$ such that
\[
 \operatorname{supp}f_j\subset K',\qquad
 \norm{f_j-f}_{L^2(dV_g)}+
 \norm{\nabla f_j-\nabla f}_{L^2(dV_g)}\longrightarrow0.
\]
If $f$ is real and nonnegative, the $f_j$ can be chosen
nonnegative. If $f$ is Lipschitz, they can also be chosen to
converge uniformly to $f$.
For every nonnegative locally bounded function $a$ and locally
bounded positive semidefinite symmetric tensor $B$,
\[
 \begin{aligned}
 &\int_M a|f_j-f|^2\dd
   +\int_M B(\nabla(f_j-f),\nabla\overline{f_j-f})\dd\longrightarrow0,\\
 &\int_M a|f_j|^2\dd\longrightarrow\int_M a|f|^2\dd,\\
 &\int_M B(\nabla f_j,\nabla\bar f_j)\dd
       \longrightarrow\int_M B(\nabla f,\nabla\bar f)\dd.
 \end{aligned}
\]
\end{proposition}

\begin{proof}
Choose finitely many relatively compact coordinate charts
$(U_i,\kappa_i)$ covering $K$, and smooth nonnegative functions
$\psi_i\in C_c^\infty(U_i)$ with $\sum_i\psi_i=1$ on a
neighborhood of $K$. Such functions are obtained from a smooth
partition of unity subordinate to these charts and $M\setminus K$;
the latter term vanishes near $K$. Consequently,
$f=\sum_i\psi_if$ almost everywhere.
For each $i$, the function $\psi_if$ has support compactly contained
in $U_i$. Its coordinate representative, extended by zero, is a
function $F_i\in W^{1,2}(\mathbb R^n)$ with compact support in
$\kappa_i(U_i)$, where $n=\dim_{\mathbb R}M$.
There is no boundary contribution to its weak derivative because
it vanishes in a neighborhood of the boundary of its chart.
On the fixed compact subsets under consideration, the metric
coefficients and volume density are uniformly comparable with
the Euclidean ones; coordinate and Riemannian $W^{1,2}$ norms
are therefore equivalent there.

Let $\varrho\in C_c^\infty(B_{\mathbb R^n}(0,1))$ be nonnegative
with integral one, and put
$\varrho_\epsilon(x)=\epsilon^{-n}\varrho(x/\epsilon)$.
Fix compact sets $K_i'$ whose interiors contain
$\operatorname{supp}F_i$ and which are contained in $\kappa_i(U_i)$.
For all sufficiently small $\epsilon>0$,
\[
 F_{i,\epsilon}:=\varrho_\epsilon*F_i\in C_c^\infty(\mathbb R^n),
 \qquad \operatorname{supp}F_{i,\epsilon}\subset K_i'.
\]
The usual convolution identity for weak derivatives gives
$\partial_aF_{i,\epsilon}=\varrho_\epsilon*(\partial_aF_i)$.
Translation continuity in $L^2(\mathbb R^n)$, applied to $F_i$
and each $\partial_aF_i$, proves
\[
 F_{i,\epsilon}\longrightarrow F_i
                   \qquad\text{in }W^{1,2}(\mathbb R^n).
\]
For example, this follows directly from
\[
 \norm{\varrho_\epsilon*G-G}_2
 \leq\int_{\mathbb R^n}\varrho(y)
          \norm{G(\,\cdot-\epsilon y)-G}_2\,dy\longrightarrow0
 \qquad(G\in L^2(\mathbb R^n)),
\]
using the bound $2\norm G_2\varrho(y)$ and dominated convergence.
Pull back $F_{i,\epsilon}$ to $U_i$ and extend by zero on $M$.
These pullbacks are smooth, since their supports lie strictly
inside their charts. Their sum, denoted by $f_\epsilon$, has
support in the fixed compact set
$K'=\bigcup_i\kappa_i^{-1}(K_i')$ and converges to $f$ in
$W^{1,2}(M)$ by the norm equivalence and the finite sum.
Choose any sufficiently small sequence $\epsilon_j\downarrow0$.
This gives the required $f_j$.

If $f\geq0$, every $F_i$ and its convolution are nonnegative,
so $f_j\geq0$. If $f$ is Lipschitz, each zero-extended $F_i$ is
continuous with compact support and hence uniformly continuous.
Thus $\varrho_\epsilon*F_i\to F_i$ uniformly; the finite sum
proves uniform convergence on $M$.

Finally, $a$ and the operator norm of $B$ are bounded on $K'$.
The first weighted convergence follows immediately from strong
$W^{1,2}$ convergence. For the quadratic limits, for example,
\[
 \left|\int_Ma(|f_j|^2-|f|^2)\dd\right|
 \leq\norm a_{L^\infty(K')}
       \norm{f_j-f}_2(\norm{f_j}_2+\norm f_2)\longrightarrow0.
\]
The identical argument with the bounded field $B^{1/2}$ and the
gradients proves the last limit.
\end{proof}

\begin{corollary}\label{cor:weak-tests}
Let $D,w$ be smooth coefficients as in Lemma~\ref{lem:operator}.
If $z\in W^{1,2}_{\mathrm{loc}}(M)$ satisfies
$-\Div(D\nabla z)+wz=0$ weakly, then
\[
 \int_M D(\nabla z,\nabla\bar\phi)\dd
       +\int_M wz\bar\phi\dd=0
\]
for every compactly supported $\phi\in W^{1,2}(M)$.
In particular, if $\chi$ is a compactly supported Lipschitz
function and $z$ is smooth, the test $\phi=\chi^2z$ is admissible.
\end{corollary}

\begin{proof}
Use Proposition~\ref{prop:compact-approximation} to approximate
$\phi$ by smooth functions with a common compact support $K'$.
On $K'$, the coefficients are bounded and $z,\nabla z$ are in
$L^2$, so Cauchy--Schwarz passes both integrals to the limit.
The final assertion follows from the weak product rule.
\end{proof}

For clarity, the same argument passes every cutoff integral used
in \eqref{eq:quadratic} and \eqref{eq:hierarchy} to the limit:
for fixed regularization parameters, their weights are locally
bounded functions of the smooth scalar curvature.
After passage to the limit one uses the gradient bound of the
Lipschitz function itself. No common pointwise gradient bound
for its smooth approximations is required. In particular, even
if these approximations extend slightly outside $B(o,t)$, the
limit gradient of $\eta_{s,t}$ vanishes almost everywhere outside
the closed annulus
$\overline{B(o,t)}\setminus B(o,s)$, as needed in
\eqref{eq:annular}. The proposition makes no assertion about
second derivatives or convergence in an operator graph norm.

\section{Mixed Ricci algebra and the Hodge--Riemann relation}\label{app:mixed-hr}

We first give the detailed calculations for
\eqref{eq:mixed-expansion}--\eqref{eq:D-formula} in
Lemma~\ref{lem:mixed-algebra}, then recall the linear
Hodge--Riemann relation and prove its uniform consequence.
All calculations are pointwise; closedness enters separately
in the divergence and energy arguments of
Section~\ref{sec:mixed-ricci}.

\subsection*{Coefficient calculations for the mixed pairing}
Fix $1\leq k\leq m-2$ and $\delta>0$. At a point choose the
Ricci eigenframe \eqref{eq:eigenvalues}, and write
\[
 A=\rho+\delta\omega=\sum_i\ell_i\alpha_i,\qquad
 \ell_i=\lambda_i+\delta,\qquad
 dV_g=\alpha_1\wedge\cdots\wedge\alpha_m.
\]
For an increasing set of indices $I$, write
$\alpha_I=\bigwedge_{i\in I}\alpha_i$ and
$\ell_I=\prod_{i\in I}\ell_i$. Denote elementary symmetric
polynomials by $\sigma_k(\ell)=\sum_{|I|=k}\ell_I$, with
$\sigma_0=1$; a vertical bar lists the variables omitted.
Since $\alpha_i^2=0$ and two-forms
commute under wedge product,
\[
 A^k=k!\sum_{|I|=k}\ell_I\alpha_I,\qquad
 \frac{\omega^{m-k-2}}{(m-k-2)!}
                  =\sum_{|J|=m-k-2}\alpha_J.
\]
For fixed $i<j$, a term contributing to the complement of
$\{i,j\}$ chooses $I\subset\{1,\ldots,m\}\setminus\{i,j\}$
of size $k$, after which $J$ is uniquely determined. Consequently,
\[
 \Theta_{k,\delta}
    =k!\sum_{i<j}\sigma_k(\ell\mid i,j)\alpha_{\widehat{i,j}},
\]
where $\alpha_{\widehat{i,j}}$ is the ordered product over all
indices other than $i,j$.

Use the unitary coframe
$\zeta_i=(\theta_i+\mathrm i\psi_i)/\sqrt2$, so that
$\alpha_i=\mathrm i\zeta_i\wedge\bar\zeta_i$, and write
\[
 U=\mathrm i\sum_{i,j}U_{i\bar j}\zeta_i\wedge\bar\zeta_j,
 \qquad
 V=\mathrm i\sum_{i,j}V_{i\bar j}\zeta_i\wedge\bar\zeta_j.
\]
The matrices of real $(1,1)$-forms are Hermitian.
The coefficient of $\alpha_i\wedge\alpha_j$ in the restriction
of $U\wedge V$ to these two complex directions is
\[
 U_{i\bar i}V_{j\bar j}+U_{j\bar j}V_{i\bar i}
       -U_{i\bar j}V_{j\bar i}-U_{j\bar i}V_{i\bar j}.
\]
The minus signs follow from
$(\mathrm i\zeta_i\wedge\bar\zeta_j)
 \wedge(\mathrm i\zeta_j\wedge\bar\zeta_i)
 =-\alpha_i\wedge\alpha_j$.
Hermitian symmetry turns the last two terms into
$-2\operatorname{Re}(U_{i\bar j}\overline{V_{i\bar j}})$.
Wedging with the complementary factors therefore gives
\begin{equation}\label{eq:app-mixed-expansion}
 \begin{aligned}
 Q(U,V)&=k!\sum_{i<j}\sigma_k(\ell\mid i,j)\\
 &\quad\cdot\bigl(U_{i\bar i}V_{j\bar j}
        +U_{j\bar j}V_{i\bar i}
        -2\operatorname{Re}(U_{i\bar j}\overline{V_{i\bar j}})\bigr).
 \end{aligned}
\end{equation}

For $U=V=\omega$, the diagonal matrix entries equal one and
the off-diagonal entries vanish. Thus
$h=2k!\sum_{i<j}\sigma_k(\ell\mid i,j)$.
Each monomial $\ell_I$ in $\sigma_k(\ell)$ occurs once for
every pair $\{i,j\}$ outside $I$, namely
$\binom{m-k}{2}$ times. Hence
\begin{equation}\label{eq:app-h-formula}
 h=2k!\binom{m-k}{2}\sigma_k(\ell)
   =k!(m-k)(m-k-1)\sigma_k(\ell).
\end{equation}

Finally, setting $V=\omega$ in \eqref{eq:app-mixed-expansion}
and collecting the coefficient of $U_{i\bar i}$ gives
\begin{equation}\label{eq:app-D-formula}
 \begin{aligned}
 Q(U,\omega)&=\sum_iD_iU_{i\bar i},\\
 D_i&=k!\sum_{j\ne i}\sigma_k(\ell\mid i,j)
       =k!(m-k-1)\sigma_k(\ell\mid i),\qquad
 \sum_iD_i=h.
 \end{aligned}
\end{equation}
Indeed, every $k$-element set $I$ avoiding $i$ is counted for
exactly $m-k-1$ choices of $j\notin I\cup\{i\}$.
Equivalently,
\[
 \omega\wedge\Theta_{k,\delta}
       =\sum_iD_i\alpha_{\widehat i}.
\]
Since the $\alpha_i$ coefficient of $df\wedge d^cu$ is
$(e_if)(e_iu)+(Je_if)(Je_iu)$, this also identifies the tensor
in Lemma~\ref{lem:mixed-algebra} as having eigenvalue $D_i$ on
both $e_i$ and $Je_i$.

\subsection*{Primitive forms and the Hodge--Riemann relation}
We now recall the linear relation on a Hermitian vector space;
no global compactness assumption on $M$ is involved.

Let $W$ have complex dimension $m\geq2$, with Hermitian form
$\omega$ and volume $dV=\omega^m/m!$. Let
$\omega_1,\ldots,\omega_{m-2},\chi$ be strictly positive real
$(1,1)$-forms, and set
\[
 \Omega=\omega_1\wedge\cdots\wedge\omega_{m-2},\qquad
 Q_\Omega(U,V)\,dV=U\wedge V\wedge\Omega.
\]
For $m=2$, the empty product $\Omega$ is $1$.
The primitive subspace relative to $(\Omega,\chi)$ is
\[
 \mathcal P_{\Omega,\chi}
 =\{U\in\Lambda^{1,1}_{\mathbb R}:
          U\wedge\chi\wedge\Omega=0\}
 =\{U:Q_\Omega(U,\chi)=0\}.
\]

\begin{theorem}[Linear mixed Hodge--Riemann relation]\label{thm:linear-mixed-hr}
With the preceding strictly positive factors, $-Q_\Omega$ is
positive definite on $\mathcal P_{\Omega,\chi}$:
\[
 -Q_\Omega(U,U)>0
 \qquad\text{for every nonzero }U\in\mathcal P_{\Omega,\chi}.
\]
\end{theorem}

This is the real $(1,1)$ case of the linear mixed
Hodge--Riemann theorem of Timorin~\cite{Timorin}, written with
the wedge-product sign convention above; see also
\cite[Proposition~2.1]{DinhNguyen}. Since $Q_\Omega(\chi,\chi)>0$, every real
$(1,1)$-form has the unique $Q_\Omega$-orthogonal decomposition
\[
 U=U_0+
       \frac{Q_\Omega(U,\chi)}{Q_\Omega(\chi,\chi)}\chi,
 \qquad U_0\in\mathcal P_{\Omega,\chi}.
\]
Thus the primitive condition removes the component in the
positive direction $\chi$.
For comparison, in the ordinary case
$\Omega=\omega^{m-2}/(m-2)!$ and $\chi=\omega$, direct expansion gives
\[
 Q_\Omega(U,V)
   =(\Lambda_\omega U)(\Lambda_\omega V)-\ip U V,
 \qquad Q_\Omega(U,\omega)=(m-1)\Lambda_\omega U.
\]
Hence ordinary primitive real $(1,1)$-forms satisfy
$\Lambda_\omega U=0$ and $-Q_\Omega(U,U)=|U|^2$.
For mixed factors the trace is instead $Q_\Omega(U,\chi)$.

\begin{corollary}[Uniform primitive estimate]\label{cor:uniform-mixed-hr}
Fix $0<\varepsilon\leq1$ and $0\leq k\leq m-2$.
For a real $(1,1)$-form $A$ with
$\varepsilon\omega\leq A\leq\omega$, define
\[
 Q_A(U,V)\,dV
   =\frac{U\wedge V\wedge A^k\wedge\omega^{m-k-2}}
           {(m-k-2)!},\qquad h_A=Q_A(\omega,\omega)>0.
\]
There are constants $c,C>0$, depending only on
$m,\varepsilon,k$, such that
\[
 \begin{aligned}
 -Q_A(U,U)&\geq c h_A|U|^2
                  &&\text{if }Q_A(U,\omega)=0,\\
 |Q_A(U,V)|&\leq C h_A|U||V|
                  &&\text{for all real }(1,1)\text{-forms }U,V.
 \end{aligned}
\]
\end{corollary}

\begin{proof}
Identify $W$ unitarily with a fixed Hermitian vector space.
The set $\mathcal K=\{A:\varepsilon\omega\leq A\leq\omega\}$
is compact. The coefficients of $Q_A$ and the strictly positive
quantity $h_A$ depend continuously on $A$.
For each $A\in\mathcal K$, apply Theorem~\ref{thm:linear-mixed-hr}
with $k$ factors equal to $A$, the other $m-k-2$ factors equal
to $\omega$, and $\chi=\omega$. The positive factorial in the
definition of $Q_A$ does not affect definiteness. Therefore
\[
 \frac{-Q_A(U,U)}{h_A}>0
 \quad\text{on}\quad
 \mathcal S=\{(A,U):A\in\mathcal K,\ |U|=1,
                                  \ Q_A(U,\omega)=0\}.
\]
The set $\mathcal S$ is compact, including its varying primitive
subspaces, because its defining constraint is continuous.
The displayed continuous positive function has a positive minimum
$c$ there. Homogeneity proves the first estimate for arbitrary
primitive $U$. Likewise, continuity on the compact set
$\mathcal K\times\{|U|=1\}\times\{|V|=1\}$ bounds
$|Q_A(U,V)|/h_A$ from above, proving the second estimate.
\end{proof}

To apply this to \eqref{eq:mixed-definitions}, at each point set
$A=(\rho+\delta\omega)/(R+\delta)$, so
$\varepsilon\omega\leq A\leq\omega$. Then
\[
 Q_{k,\delta}=(R+\delta)^kQ_A,\qquad
 h=(R+\delta)^kh_A.
\]
The corollary gives the primitive coercivity uniformly in the
point and in $\delta>0$. In particular, for
$E=\rho-a\omega$ from \eqref{eq:mixed-trace},
$Q_{k,\delta}(E,\omega)=0$ and
$-Q_{k,\delta}(E,E)\geq ch|E|^2$.
This normalization is used only in tangent-space algebra; the
normalized $A$ is not substituted for a closed differential form
in Stokes' theorem.

\end{document}